\documentclass[journal]{IEEEtran}
\usepackage[T1]{fontenc}
\usepackage[latin9]{inputenc}
\usepackage{prettyref}

\usepackage{amsmath, amssymb}
\usepackage{amsthm}
\usepackage{graphicx}
\PassOptionsToPackage{normalem}{ulem}
\usepackage{ulem}
\usepackage{cancel}
\usepackage{tabularx}
\usepackage[unicode=true,
bookmarks=true,bookmarksnumbered=true,bookmarksopen=true,bookmarksopenlevel=1,
 breaklinks=false,pdfborder={0 0 0},pdfborderstyle={},backref=false,colorlinks=false]
 {hyperref}
\hypersetup{pdftitle={Your Title},
 pdfauthor={Your Name},
 pdfpagelayout=OneColumn, pdfnewwindow=true, pdfstartview=XYZ, plainpages=false}

\makeatletter

\usepackage[caption=false,font=footnotesize]{subfig}

\usepackage{cite}
\usepackage{eurosym}

\usepackage{hyperref}

\usepackage{soul}
\soulregister\cite7
\soulregister\ref7
\soulregister\eqref7
\usepackage[dvipsnames]{xcolor}
\DeclareRobustCommand{\hlyellow}[1]{{\sethlcolor{white}\hl{#1}}} 
\DeclareRobustCommand{\hlorange}[1]{{\sethlcolor{white}\hl{#1}}} 
\DeclareRobustCommand{\hlgreen}[1]{{\sethlcolor{white}\hl{#1}}} 
\DeclareRobustCommand{\hlblue}[1]{{\sethlcolor{white}\hl{#1}}} 
\DeclareRobustCommand{\hlpink}[1]{{\sethlcolor{white}\hl{#1}}} 
\DeclareRobustCommand{\hlteal}[1]{{\sethlcolor{white}\hl{#1}}} 
\DeclareRobustCommand{\hlinv}[1]{{\sethlcolor{white}\hl{#1}}} 

\allowdisplaybreaks

\newtheorem{theorem}{Theorem}[section]
\newtheorem{lemma}[theorem]{Lemma}
\newtheorem{proposition}[theorem]{Proposition}

\def\fnum@figure{Fig.~\thefigure}

\usepackage{multirow,colortbl}
\usepackage{psfrag}

\usepackage{soul}
\soulregister\cite7
\soulregister\ref7
\soulregister\eqref7

\usepackage[colorinlistoftodos]{todonotes}

\newrefformat{fig}{Fig.~\ref{#1}}
\newrefformat{sec}{Section~\ref{#1}}
\newrefformat{subsec}{Section~\ref{#1}}
\newrefformat{app}{Appendix}
\newrefformat{tab}{Table~\ref{#1}}

\newcommand{\RU}{\overline{R}^{\mathrm{up}}}
\newcommand{\RD}{\overline{R}^{\mathrm{dn}}}
\newcommand{\SUC}{C^{\mathrm{su}}}
\newcommand{\SDC}{C^{\mathrm{sd}}}
\newcommand{\Cp}{C^{\mathrm{p}}}
\newcommand{\Cu}{C^{\mathrm{nl}}}

\newcommand{\Pup}{\overline{P}}
\newcommand{\Plow}{\underline{P}}
\newcommand{\csu}{c^{\mathrm{su}}}
\newcommand{\csd}{c^{\mathrm{sd}}}

\newcommand{\mut}{\underline{T}^\mathrm{up}}
\newcommand{\mdt}{\underline{T}^\mathrm{dn}}
\newcommand{\SU}{\overline{P}^\mathrm{su}}
\newcommand{\SD}{\overline{P}^\mathrm{sd}}
\newcommand{\Psu}{P^\mathrm{su}}
\newcommand{\Psd}{P^\mathrm{sd}}
\newcommand{\sut}{T^\mathrm{su}}
\newcommand{\sdt}{T^\mathrm{sd}}
\newcommand{\tstart}{t_0}
\newcommand{\tend}{t_{end}}
\newcommand{\Tzero}{\mathcal{T}^0}
\newcommand{\Tmut}{\mathcal{T}^\mathrm{up}}
\newcommand{\Tmdt}{\mathcal{T}^\mathrm{dn}}

\newcommand{\uinv}{u^\mathrm{inv}}

\makeatother

\begin{document}
\title{Tight Formulations for Unit Commitment with Different Levels of Details -- Part I: Model\hlpink{s }and Theoretical Insights}
\author{$\negthickspace\negthickspace\negthickspace\negthickspace\negthickspace\negthickspace\negthickspace\negthickspace\negthickspace\negthickspace$Maaike B. Elgersma,\hlteal{ }Karen I. Aardal, Mathijs M. de Weerdt, and~\hlteal{Germ\'{a}n Morales-Espa\~{n}a}
$\negthickspace\negthickspace\negthickspace\negthickspace\negthickspace\negthickspace\negthickspace\negthickspace\negthickspace\negthickspace$}
\maketitle
\begin{abstract}
The unit commitment (UC) problem is \hlyellow{paramount for optimal} operation of power systems, but it faces computational limitations in large-scale \hlteal{settings, especially in} \hlpink{investment or stochastic models,} because of the binary variables that it contains. A lot of research has attempted to improve the computational performance of\hlpink{ }UC models, either by reducing model \hlpink{size, resulting in lower }fidelity and accuracy, or by improving the tightness of the formulation. Tightness and model size are the best a priori indicators of the computational performance of UC models, but there is no clear overview of what the best formulation is for different \hlpink{generators}. In this research, we define models with different levels of detail, and present a formulation \hlpink{for each level} that is based on the convex hull\hlyellow{.} We show new \hlpink{proofs} on the tightness of well-known formulations \hlblue{for ramping, for start-up and shut-down costs and capabilities, and for UC with investment.}
\hlteal{These models, with a different level of detail, can be incorporated into large-scale problems to reduce the computational burden, as demonstrated in Part II.} 
\end{abstract}


\section*{Nomenclature}
An overview of the notation used throughout this paper is given below.

\vspace{5pt}
\noindent \textit{Sets and indices:}
\begin{IEEEdescription}[\IEEEusemathlabelsep\IEEEsetlabelwidth{$\Tmut_g$}]
    \item[$g\in \mathcal{G}$] units (generators)
    \item[$\mathcal{G}^1$] units (generators) with a minimum up time of 1 ($\subseteq \mathcal{G}$)
    \item[$t\in \mathcal{T}$] time periods, where $\tstart$ is the first
    \item[$\Tzero$] $=\mathcal{T}\setminus\{\tstart\}$
    \item[$\Tmut_g$] $=\mathcal{T}\setminus\{\tstart,\dots,\tstart+\mut_g\}$
    \item[$\Tmut_g$] $=\mathcal{T}\setminus\{\tstart,\dots,\tstart+\mdt_g\}$
\end{IEEEdescription}
\noindent \textit{Parameters:}
\begin{IEEEdescription}[\IEEEusemathlabelsep\IEEEsetlabelwidth{$\Cp_g$}]
\item[$\Cu_g$] \hlyellow{cost of non-load of unit $g$}
\item[$\Cp_g$] cost of energy production of unit $g$
    \item[$\SUC_g$] cost of starting up unit $g$ once
    \item[$\SDC_g$] cost of shutting down unit $g$ once
    \item[$D_t$] demand in time period $t$
    \item[$\Psu_{gi}$] output of unit $g$ in time period $i$ of the start-up trajectory 
    \item[$\Psd_{gi}$] output of unit $g$ in time period $i$ of the shut-down trajectory 
    \item[$\Plow_g$] minimum output of unit $g$ in one time period
    \item[$\Pup_g$] maximum output of unit $g$ in one time period
    \item[$\SU_g$] maximum output of unit $g$ when it starts up
    \item[$\SD_g$] maximum output of unit $g$ when it shuts down
    \item[$\RU_g$] maximum amount that unit $g$ can ramp-up its output in one time period
    \item[$\RD_g$] maximum amount that unit $g$ can ramp-down its output in one time period
    \item[$\mut_g$] minimum up time of unit $g$
    \item[$\mdt_g$] minimum down time of unit $g$
    \item[$\sut_g$] duration of start-up trajectory of unit $g$
    \item[$\sdt_g$] duration of start-up trajectory of unit $g$
\end{IEEEdescription}
\noindent \textit{Variables:}
\begin{IEEEdescription}[\IEEEusemathlabelsep\IEEEsetlabelwidth{$\Cp_g$}]
    \item[$\csu_{gt}$] costs associated with starting up unit $g$ in time period $t$ 
    \item[$\csd_{gt}$] costs associated with shutting down unit $g$ in time period $t$
    \item[$p_{gt}$] amount of energy that unit $g$ produces in time period $t$
    \item[$p^\mathrm{traj}_{gt}$] amount of energy that unit $g$ produces in its start-up or shut-down trajectory in time period $t$
    \item[$u_{gt}$] indicates whether unit $g$ is turned on or off in time period $t$
    \item[$\uinv_{g}$] \hlblue{indicates the number of investments in unit $g$}
    \item[$v_{gt}$] indicates whether unit $g$ is starts up in time period $t$ or not
    \item[$w_{gt}$] indicates whether unit $g$ is shuts down in time period $t$ or not
\end{IEEEdescription}

\section{Introduction}

\IEEEPARstart{T}{he} \hlteal{unit commitment (UC) problem is one of the most important problems for power system management}~\cite{monteroReviewUnitCommitment2022,ostrowskiTightMixedInteger2012}. It is a traditional optimization problem that obtains the best operational schedule for a group of units, such as thermal generators, nuclear power plants, and renewable generators, but it can also be used for other \hlyellow{types of units}, such as electrolyzers~\cite{gongJointUnitCommitment2024} and block bids~\cite{jianComparisonBlockBidding2002}. It \hlyellow{is typically formulated as a Mixed Integer Linear Program (MILP) that} minimizes system-wide operational costs of power generators \hlyellow{for 24-48 hours}, while making sure that the demand is met \hlpink{and operational capabilities of the units (the UC constraints) are respected}, but it faces computational limitations in large-scale applications~\cite{monteroReviewUnitCommitment2022,ostrowskiTightMixedInteger2012}. \hlyellow{Moreover, the increasing share of variable renewable energy has strongly increased the level of uncertainty in the system. Ideally, a stochastic or robust version of the UC problem would be considered, but this is an even more challenging problem to \mbox{solve~\cite{vanackooijLargescaleUnitCommitment2018}}.} As a result \hlyellow{of these computational challenges}, UC constraints are typically not considered \hlpink{at all} in large-scale investment models~\cite{wuijtsEffectModellingChoices2023}. However, omitting short-term operational details \hlpink{entirely} can lead to substantial overestimation of system flexibility\hlteal{, resulting in unfeasible solutions in practice,} and underestimation of system costs~\cite{morales-espanaHiddenPowerSystem2017}. Although such conclusions often heavily depend on the case study, several studies agree that the increase of variable renewable energy system integration increases the importance of considering these operational details~\cite{ponceletImpactLevelTemporal2016,palmintierImpactUnitCommitment2011,huaRepresentingOperationalFlexibility2018a}. This highlights the need for \hlpink{different,} computationally efficient ways to model and solve UC problems\hlpink{, so they can be at least partially incorporated into investment models while staying within computational limits}. Therefore, we look into ways to include UC models with different levels of detail into large-scale models. \hlpink{For different contexts and problem scales, the UC model with the most suitable level of detail (likely the highest possible) can be selected and included.} For each of these levels, we only consider formulations with strong LP relaxations that do not substantially increase model size.
\hlyellow{ }

The most widespread approach to model the UC problem is by addressing it as a classical optimization problem, according to Montero et al.~\cite{monteroReviewUnitCommitment2022}. There are many different types of constraints that can be considered, such as capacity limits, ramping limits, and minimum up and down times, all of which can be modeled in different ways. One of the most \hlyellow{common} ways to model UC problems is MILP~\cite{monteroReviewUnitCommitment2022}. The number of binary variables per time periods varies between different formulations, related to the type of constraints that are modeled~\cite{ridhaComplexityProfilesLargeScale2020}. To solve the models, several decomposition methods and heuristics, such as evolutionary algorithms, can be used, but commercial solvers are typically the most advanced and \hlyellow{commonly used} tools~\cite{monteroReviewUnitCommitment2022}. However, because of the large number of binary variables, even the best solvers are often unable to solve large-scale investment problems with UC in the desired time frame.\hlpink{ }The linear relaxation is \hlpink{then} solved instead~\cite{ridhaComplexityProfilesLargeScale2020}\hlorange{, which is obtained by relaxing the integrality constraints}, but this could introduce large errors.

A lot of research has attempted to improve the computational performance of UC models by reducing model fidelity in some dimension. Some have focused on clustering similar units, resulting in cluster UC models with fewer integer variables, rather than binary variables for each unit~\cite{palmintierImpactOperationalFlexibility2015,palmintierHeterogeneousUnitClustering2014,zhangTightUnitAggregation2025,huaRepresentingOperationalFlexibility2018a,wangLongcycleStoragesAligned2023}. However, these clustered formulations can introduce errors by overestimating the capabilities of individual units~\cite{meusApplicabilityClusteredUnit2018,morales-espanaModelingHiddenFlexibility2019}. Other research has focused on representing the temporal data with a smaller set, such as using representative days~\cite{helistoImpactOperationalDetails2021,ponceletImpactLevelTemporal2016,ponceletSelectingRepresentativeDays2017,wogrinWhatTimeperiodAggregation2019}. This reduces the size of models significantly, but can also introduce errors by over- or underestimating the needed investments or intertemporal capabilities of generators. 

A promising direction for improving solution times without sacrificing accuracy has been to obtain tight MILP formulations for UC~\cite{damci-kurtPolyhedralStudyProduction2016,morales-espanaTightCompactMILP2013,morales-espanaTightCompactMILP2013a,rajanMinimumPolytopesUnit2005,gentileTightMIPFormulation2017a,ostrowskiTightMixedInteger2012,queyranneTightMIPFormulations2017,yangNovelProjectedTwo2016}. 
The concept of tightness relates to the gap between the \hlyellow{mixed-integer} feasible region of the\hlyellow{ }model and that of its LP relaxation. The tightest formulation of a MILP is also referred to as the \hlorange{`convex hull of feasible solutions', or `convex hull' in short}. \hlorange{We refer to the facet defining constraints that describe the convex hull as `facets'}. Replacing original constraints by facets results in `tighter' formulations\hlyellow{.}
In practice, models based on the convex hull formulation -- even those which contain only a subset of the facets -- often have a significantly shorter solution time than less tight models, as shown in~\cite{tejada-arangoWhichUnitCommitmentFormulation2020,morales-espanaTightMIPFormulations2015a}.\hlyellow{ } 

However, finding the full convex hull is not straightforward, and it may consist of exponentially many facets\hlpink{.} For example, Lee et al.~\cite{leeMinminpolytopes2004} proved that the number of facets of the convex hull of the UC problem with minimum up and down times is exponential with respect to the number of time periods. Thereafter, Rajan and Takriti~\cite{rajanMinimumPolytopesUnit2005} presented a higher-dimensional formulation and proved that the size of its convex hull grows linearly. The second formulation performed much better in experiments. Specific separation algorithms can be designed to include an exponential number of constraints, but the benefits disappear \hlyellow{for} more complex versions of the UC problem\hlyellow{.} Thus, \hlpink{including additional facets improves the tightness, but the increased model size could negatively impact the solution }\hlteal{time}\hlpink{, so} there is often a trade-off between the tightness and size of a MILP formulation. 

For some UC problems, it is not clear a priori (before solving) which formulation has the ideal balance between tightness and size, such that it is solved the fastest. For example, for start-up and shut-down capability constraints, there exist several formulations that are similar in terms of tightness and size, but it has not been investigated how they compare exactly. The computational performance of a UC formulation can only be determined a posteriori in experiments, but it depends heavily on the case study, the solver, and randomness in the solver, and it could change with each solver update. For example, Knueven et al.~\cite{knuevenMixedIntegerProgramming} computationally compared several UC formulations for different instances in three test sets. They observed that it differed per test set and even per instance and run which formulation\hlyellow{ }solved the fastest. \hlpink{Moreover, a posteriori comparison is often not possible for real-life problem instances because of computational limitations.} 

\hlpink{Thus, a posteriori comparison is not always an option, and its conclusions are not necessarily generalizable.}
\hlpink{For a priori model selection}, we argue that the metrics of tightness and size are the most informative indicators of formulation quality, but there is often a trade-off between the two. Therefore, in this paper, we choose to only consider UC formulations that are based on the convex hull of a subproblem, but similar in size to the smallest possible valid formulation. More specifically, if the convex hull contains exponentially many constraints, we only consider a subset of the facets, which forms a valid formulation and is similar in size to the smallest valid formulation. Furthermore, we introduce \hlyellow{UC} formulations with different levels of\hlyellow{ }detail, rather than only the most detailed model.

The specific contributions of \hlpink{Part I and II of} this \hlyellow{paper} are the following:
\begin{enumerate}
    \item We present a structured overview of UC models with different levels of\hlyellow{ }detail. For each level, we present one or several formulations that meet our chosen trade-off between tightness and size: the tightest formulation of a similar size as the smallest valid formulation.
    \item We provide\hlyellow{ }theoretical insights \hlorange{into} \hlyellow{the tightness of often-used formulations for ramping, for start-up and shut-down costs and capabilities (summarized in Tables~\ref{tab:convexhulls1bin} and~\ref{tab:convexhulls})}\hlblue{, and for }\hlteal{UC within investment problems}\hlblue{ in Section~\ref{sec:inv}.}
    \item We illustrate how incorporating the formulations with different levels of detail in large-scale investment and operational models affects their accuracy and computational performance.
    \item We directly compare, both theoretically and numerically, different formulations for the same problem, which are similar in terms of tightness and size, resulting in new insights into how they perform. 
\end{enumerate}

With these contributions, we aim to help modelers find a suitable UC formulation with a\hlyellow{n adequate} level of detail for their\hlyellow{ }problems. We provide contributions 1) and 2) in Part I of this paper, and the others in Part II.

    
The rest of the paper is structured as follows.\hlyellow{ }Section~\ref{sec:uc}\hlyellow{ }introduces the basic models and concepts, and give\hlyellow{s} a structured overview of the different models that we present in this paper. We then introduce the different model formulations in Section~\ref{sec:ucformulations}, from the most basic to the most extended, and provide proofs on their tightness. 


\section{Background UC Modeling}
\label{sec:uc}
\hlyellow{In this section, we introduce the most basic UC models, provide a clear explanation of the different capabilities, and present a structured overview of the UC models with different levels of detail.}


\subsection{Basic unit commitment}
\label{sec:basicucmodel}

\hlteal{UC}\hlblue{ models consist of constraints that describe the operational capabilities of units. In this paper, we discuss how these operational models can be incorporated into larger energy system optimization models }\hlteal{(ESOMs)}\hlblue{, such as the one presented in~\eqref{eq:esom}. Its objective~\eqref{eq:objesom} minimizes investment costs and operational costs, where $f^\mathrm{inv}$ and $f^\mathrm{p}$ are linear cost functions, whilst satisfying demand in constraint~\eqref{eq:balance}. Constraint~\eqref{eq:inv} enforces investment limits.} 

\noindent\hrulefill \\
\hlblue{\textbf{ESOM}: Energy System Optimization Model} \\
\vspace{-5pt}
\noindent\hrule
\vspace{-5pt}
\begin{subequations}
\label{eq:esom}
\begin{align}
    \min \quad & f^\mathrm{inv}(\uinv_g) + f^\mathrm{p} ( p_{gt},u_{gt}\textcolor{gray}{,v_{gt},w_{gt}}) \hspace{-60pt}& \label{eq:objesom}\\
    \text{s.t.}\quad &\sum_{g\in\mathcal{G}} p_{gt} = D_t \qquad&\forall t\in \mathcal{T} \label{eq:balance}\\
    &u_{gt} \leq \uinv_g \leq 1 \qquad&\forall g\in \mathcal{G},t\in \mathcal{T} \label{eq:inv} \\
    &p_{gt}\in \mathbb{R}_{\geq 0} \qquad&\forall g\in \mathcal{G},t\in \mathcal{T} \label{eq:p}\\
    &\uinv_g\in \mathbb{Z}_{\geq0} \qquad&\forall g\in \mathcal{G} \label{eq:uinv}\\
    &\textit{<UC constraints here>} \nonumber& 
\end{align}
\end{subequations}
\vspace{-10pt}
\noindent\hrule 
\vspace{10pt} 

\hlyellow{W}e present \hlblue{UC models with different levels of detail that can be included in such ESOMs, starting with} the \hlpink{two} most basic UC models\hlpink{, namely \mbox{\hyperref[eq:1binucmodel]{Model I}} and \mbox{\hyperref[eq:3binucmodel]{Model II}}}. We first prove everything for an operational problem, and leave out the the $\uinv$ variable.
\hlblue{However, in Section~\ref{sec:inv} we show that all theoretical insights do hold for investment problems, where constraints~\eqref{eq:inv} and~\eqref{eq:uinv} are added.}

It is important to note that all theoretical insights into the UC models presented in this paper only hold \hlteal{for a single-unit problem, }without considering the balance constraint~\eqref{eq:balance}. The objective is usually irrelevant when considering the convex hull, unless mentioned otherwise. Therefore, \hlblue{we} leave out the objective and balance constraint in the model formulations. \hlblue{To further simplify notation, we drop the $g$ subscript.}

\hyperref[eq:1binucmodel]{Model I} presents the simplest UC constraints, \hlblue{which enforce generator limits~\eqref{eq:pboundup} and~\eqref{eq:pboundslow}}. The objective typically minimizes the total production costs \hlblue{and non-load costs: $f^\mathrm{p}(p_{t},u_{t})=\sum_{t\in\mathcal{T}} \left( \Cp p_t + \Cu u_{t} \right)$.}

\noindent\hrulefill \\
\textbf{Model I}: Basic unit commitment with generation limits \\
\vspace{-5pt}
\noindent\hrule
\vspace{-5pt}
\begin{subequations}
\label{eq:1binucmodel}
\begin{align}
    &p_{t} \leq \Pup  u_{t} \qquad&\forall t\in \mathcal{T} \label{eq:pboundup}\\
    &p_{t} \geq \Plow  u_{t} \qquad&\forall t\in \mathcal{T} \label{eq:pboundslow}\\
    &u_{t} \in \{0,1\} \qquad&\forall t\in \mathcal{T} \label{eq:u}
\end{align}
\end{subequations}
\vspace{-10pt}
\noindent\hrule 
\vspace{10pt} 

There are often more limits on the production,\hlyellow{ }such as ramping limits, starting up limits, etc. Additional variables can be introduced to help model these properties, namely start-up variable $v_{gt}$ and shut-down variable $w_{gt}$. A basic UC model that includes these variables is given in \hyperref[eq:3binucmodel]{Model II}.\footnote{
\hlblue{When considering an investment problem, constraint~\eqref{eq:wcommitment} should be replaced by $w_t\leq\uinv - u_t$. We elaborate on theoretical results for this extension to investment problems in Section~\ref{sec:inv}.}} 
The objective \hlblue{usually includes start-up and shut-down costs: $f^\mathrm{p}(p_{t},u_{t},v_t,w_t)=\sum_{t\in\mathcal{T}} \left( \Cp p_t + \Cu u_{t} + \SUC_g  v_{t} + \SDC_g  w_{t} \right)$}. \hlgreen{In the remainder of this paper, we assume that all costs are nonnegative.}

\noindent\hrulefill \\
\textbf{Model II}: Generation limits and start-up \& shut-down costs \\
\vspace{-5pt}
\noindent\hrule
\vspace{-5pt}
\begin{subequations}
\label{eq:3binucmodel}
\begin{align}
    &\eqref{eq:pboundup}-\eqref{eq:u} \notag \\
    &u_{t} - u_{t-1} = v_{t} - w_{t} \qquad&\forall t\in \Tzero \label{eq:commitment}\\
    &v_{t} \leq u_{t} \qquad&\forall t\in \Tzero \label{eq:vcommitment}\\
    &w_{t} \leq 1 - u_{t} \qquad&\forall t\in \Tzero \label{eq:wcommitment}\\
    &v_{t} \in \{0,1\} \qquad&\forall t\in \Tzero \label{eq:v}\\
    &w_{t} \in \{0,1\} \qquad&\forall t\in \Tzero \label{eq:w}
\end{align}
\end{subequations}
\vspace{-10pt}
\noindent\hrule 
\vspace{10pt} 

In the next section, we explain all other \hlyellow{constraints} that can be added to these basic models.  

\subsection{Generator properties}
\label{sec:notation}
The energy production of \hlyellow{slower}, thermal\hlyellow{ }generators is \hlpink{also} bounded by \hlpink{ramping, start-up and shut-down, and minimum up and down time constraints. In this section, we illustrate these properties with an example and explain them in detail}. \hlyellow{Figure~\ref{fig:capabilitiesexample} shows} an example of an output trajectory of\hlpink{ }a unit over 12 hours, as well as the corresponding variable values. 

\begin{figure}[h!]
    \centering
    \includegraphics[width=\linewidth]{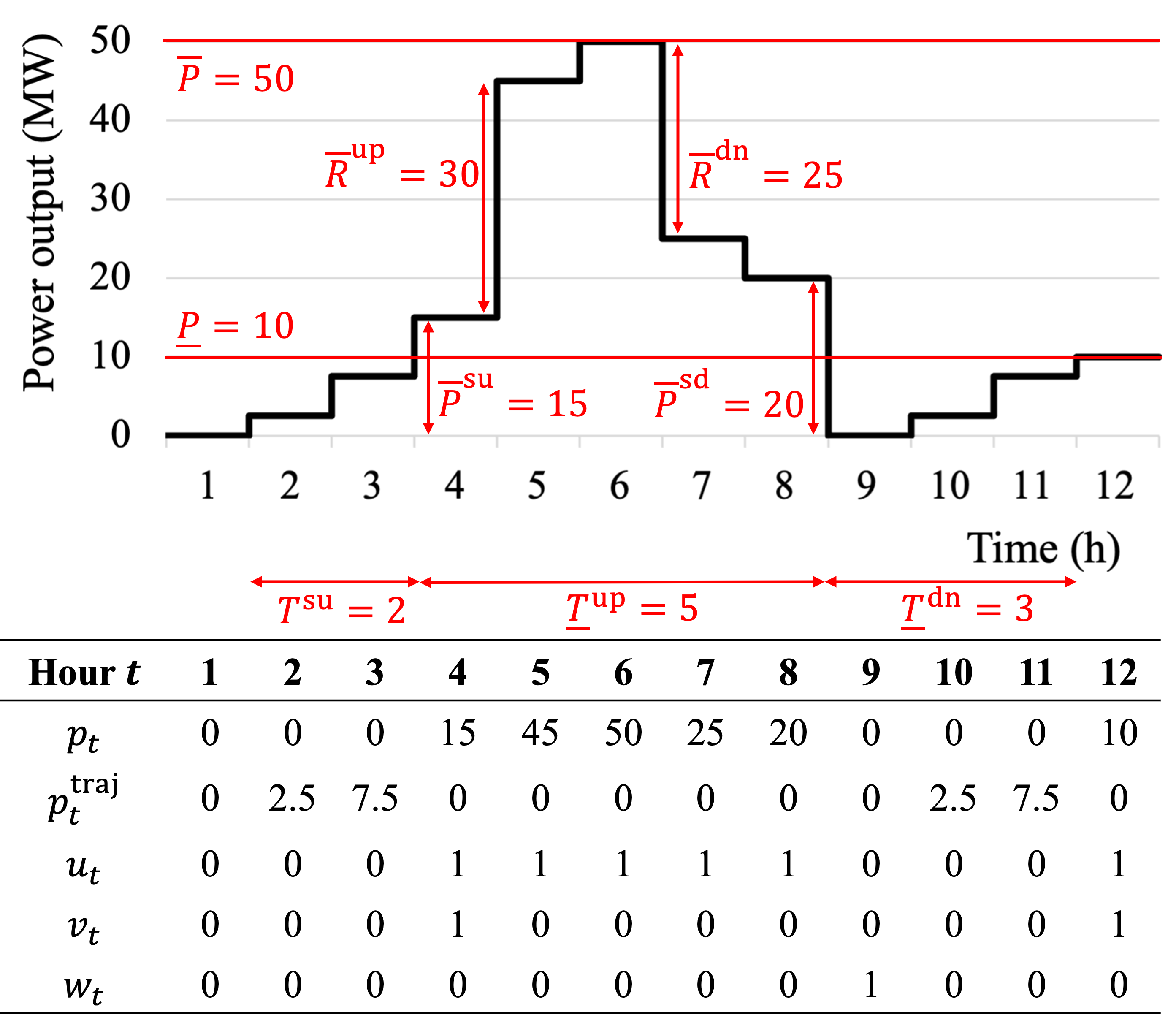}
    \caption{An example of an output trajectory of a unit over 12 hours (black line), its capabilities/limits (red annotations), and corresponding variable values.}
    \label{fig:capabilitiesexample}
\end{figure}

The unit produces between its minimum and maximum limits in hours 4 to 8. This production is also bounded by its ramping \hlyellow{limits} $\RU$ and $\RD$ from one hour to the next. The corresponding variable $u_t$ is 1 in those hours, indicating that the unit is \textit{up}. Note that this is not the same as the unit being \textit{on}. The unit already turns \textit{on} in hour 2, but it has a slow start-up trajectory of two hours. The power output during this trajectory is modeled by a separate variable $p^{\mathrm{traj}}_t$. 
Similarly, if a unit has a shut-down trajectory, we distinguish between the unit being \textit{down} or \textit{off}. 

The example also shows that the unit has a different start-up and shut-down \textit{capability}. This is the maximum amount that a unit can generate in the hour it starts \textit{up} (hour 4 here) and in the hour before it shuts \textit{down} (hour 8 here), respectively. Logically, these capabilities \hlorange{$\SU$} and $\SD$ are always at least $\Plow$ and at most $\Plow + \RU$ and $\Plow + \RD$, respectively. Start-up variable $v_t$ and shut-down variable $w_t$ are 1 in the hour that the unit starts up and shuts down, respectively.

The last property that can limit production is minimum \textit{up} and \textit{down} time. This relates to the minimum number of consecutive hours $u_t$ should be 1 or 0, respectively. If the unit has a start-up trajectory and/or shut-down trajectory, then its minimum downtime is at least as long as the duration of those trajectories together, so $\mdt\geq \sut + \sdt$.



\begin{figure*}[h!]
    \centering
    \includegraphics[width=\linewidth]{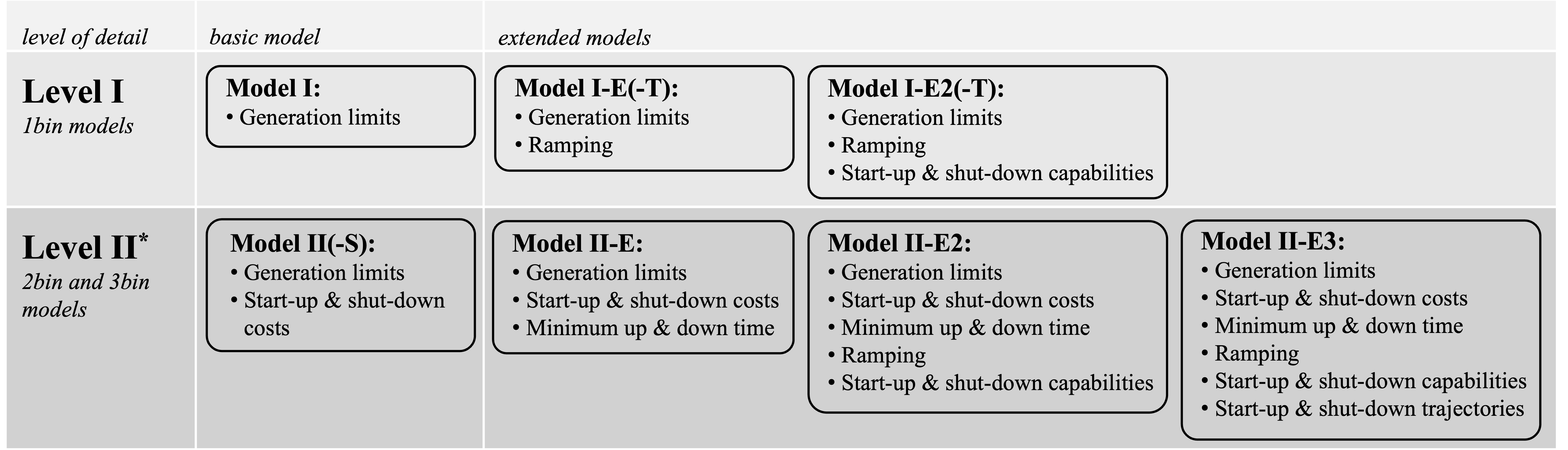}
    \caption{An overview of different levels of detail, the basic and extended models in each level of detail, and the constraints of generators that each model captures. For some, we also present a tighter (T) or smaller (S) model. \\
    $^{*}$ All models in this level can be written as a 2bin model or as an equivalently tight 3bin model.}
    \label{fig:levelsofdetail}
\end{figure*}

\subsection{Defining Different Levels of Detail}
\label{sec:defininglevels}

We now present a structured overview of different types of UC models, that have different levels of detail. The overview can help to select the appropriate model for a certain type of generator, or to select the highest level of detail that is computationally possible. To this end, we first divide UC models into the following two categories, depending on the number of binary variables:

     \textbf{1bin:} UC models containing one binary variable per time period for each generator, representing its commitment (up/down) status. 
     
    \textbf{2bin or 3bin:} UC model containing additional variables representing start-up and shut-down of the unit. We denote models that contain both by \textit{3bin}, and start-up only by \textit{2bin}. For some properties, such as minimum up and down times, introducing these additional variables can improve tightness and/or size of the model. Note that the shut-down variable is defined by equality constraint~\eqref{eq:commitment}, so every 3bin model can be directly rewritten \hlorange{as} a 2bin model. These models are equal in terms of tightness, but the number of variables and (non-zero entries in) constraints\hlyellow{ }differs. Furthermore, these start-up and shut-down variables are usually defined as binary, but they can technically be relaxed \hlorange{to $0\leq\hdots\leq 1$} without changing the set of feasible solutions to the MIP model, which we explain further in Section~\ref{sec:susdvariables}. This results in a model with just one set of binary variables. However, to differentiate them from 1bin models, we classify every model that contains start-up and shut-down variables as 2bin/3bin, regardless of whether the variables are defined as binary or continuous.


Within these two categories, we now define several UC models with different levels of detail. \hlyellow{Figure~\ref{fig:levelsofdetail} presents} a schematic overview of the models and the generator properties that they model. For each of the two categories, we present a basic model and extensions of that model. \hlteal{We argue that most generators can be accurately modeled using one of these models.} 

\subsection{Selecting the appropriate model}

In practice, models of all different levels of detail can be employed for different purposes, applications, and generators. A modeler can use the overview in Figure~\ref{fig:levelsofdetail} to select the model with the desired level of detail for a specific generator. For example, when one desires to model minimum up and down times or start-up and shut-down trajectories, it becomes clear from the overview that a 2bin or 3bin model is \hlyellow{advised}. Contrary to common \hlorange{assumptions}, we argue that a 2bin/3bin model is needed for including start-up and shut-down costs as well, and explain this further in Section~\ref{sec:susdcosts}. Furthermore, if the selected UC model results in a longer solution time than desired, a model with a lower level of detail can be selected instead, while accepting a\hlyellow{ possible overestimation of the flexibility of a unit} due to simplification.

However, in some cases, it might not be obvious which model is the most suitable. For example, when one wants to model ramping and start-up and shut-down capabilities, the 1bin Model I-E2 can be used, but \hlyellow{we present} two different formulations for this model. One is tighter (denoted by T), \hlyellow{while} the other \hlteal{leaves out the facets that are not affecting the integer feasible region, so it has fewer constraints,} and it is unclear \hlyellow{a priori} which one performs better computationally. Moreover, 3bin Model II-E2 could also be used to model this problem, by setting the start-up costs to 0 and the minimum up and down time to 1.\hlyellow{ } 
Therefore, we \hlyellow{focus on formulations that are based on the convex hull, and analyze their tightness in this paper}. In Part II, we compare the models from Figure~\ref{fig:levelsofdetail}\hlyellow{ }numerically. 

\section{UC Model Formulations}
\label{sec:ucformulations}
For each of the models in Figure~\ref{fig:levelsofdetail}, we present a formulation that is based on the convex hull of the \hlorange{problem at hand}, but similar in size to the smallest valid formulation. Therefore, we often begin by discussing the simplest valid formulation, and thereafter present the findings on the convex hull (by us or from the literature) that we use. We remind the reader that all theoretical insights only hold for the single-unit problem, without considering the energy balance constraint. \hlblue{Extensions of these theoretical insights to investment problems are presented at the end, in Section~\ref{sec:inv}.}

\subsection{Level I: 1bin UC Models}
\label{sec:1bin}
In this section, we consider 1bin UC models that only contain commitment variables as binary variables\hlorange{, starting with the most basic one}. 
Thereafter, we present some new facets for ramping down and shut-down capabilities, as well as some new proofs on the convex hull of these properties.\hlorange{ }

\subsubsection{Generation limits (1bin)}
\label{sec:generationlimits}

The most basic 1bin model containing only generation limits and linear production costs is \hyperref[eq:1binucmodel]{Model I}\hlyellow{, as presented previously in Section~\ref{sec:basicucmodel}}. Garver~\cite{garverPowerGenerationScheduling1962}, who was the first to model \hlteal{UC} in 1962, included piecewise linear production costs in his original formulation, rather than regular linear costs. Others, such as Knueven et al.~\cite{knuevenMixedIntegerProgramming}, further tightened this formulation, but all these formulations require additional variables, so we do not consider \hlorange{these costs}. 


The LP relaxation of this model is obtained by relaxing the integrality constraint~\eqref{eq:u}, i.e., $0 \leq u_{t} \leq 1$. In fact, only the constraint $u_{t}\leq 1$ is needed in the LP relaxation, since constraints~\eqref{eq:pboundup} and~\eqref{eq:pboundslow} together imply that $u_{t}\geq 0$. 
This simplified LP relaxation actually \hlgreen{defines} the convex hull of the feasible solutions of this problem. Therefore, $u_{t}$ is integer in all extreme points of the feasible space. This follows directly from Lemma~\ref{thm:thelemma}, by considering a polytope $P$ consisting of the commitment variables $u_{t}$ and constraints $0\leq u_{t}\leq 1$, and adding the production variables $p_{t}$ with constraints~\eqref{eq:pboundup} and~\eqref{eq:pboundslow}. 

\begin{lemma}
    Suppose that $P = \{x\in \mathbb{R}^n : Ax \leq b\}$ with $A\in \mathbb{R}^{n\times m}$ and $b\in \mathbb{R}^m$ is an integer \hlorange{polytope}, i.e., $P = conv(P \cap \mathbb{Z}^n)$\hlorange{ where $conv(S)$ denotes the convex hull of set S}. We now introduce $k$ new variables and $2k$ new constraints, such that we obtain the new \hlorange{polytope} $P' = \{ (x,y) \in \mathbb{R}^{n}\times\mathbb{R}^{k} : Ax \leq b, Cx \leq y \leq D x \}$, with $C, D \in \mathbb{R}^{k \times n}$. If $Cx \leq Dx$ holds for all $x\in P$, then $x$ is integer in all extreme points of $P'$. 
    \label{thm:thelemma}
\end{lemma}
\begin{proof}
    First, we observe that the $2k$ new constraints intersect in the origin. Thus, $x^*$ in this vertex $(x^*,y^*)=(0,0)$ is integer. Therefore, we only need to consider the intersection between the $2k$ new constraints and the original constraints of $P$, and verify that $x$ is integer in these vertices. To this end, we obtain the projection of $P'$ onto $y=0$, so a projection from the $(x,y)$-space to the $x$-space, denoted by $\text{proj}_x(P')$, by performing Fourier-Motzkin elimination on the new $y$ variables. The lower bounds on the $y$ variables are $y\geq Cx$, and the upper bounds are $y\leq Dx$. Combining all upper and lower bounds on $y$ to eliminate the variables gives us the constraint set $Cx \leq Dx$, so $\text{proj}_x(P') = \{ x\in \mathbb{R}^n : Ax \leq b, Cx \leq D x \}$. We already assumed the constraint set $Cx \leq Dx$ to hold for all $x\in P = \{x\in \mathbb{R}^n : Ax \leq b\}$, so these constraints are redundant, and the projection becomes $\text{proj}_x(P') = \{ x\in \mathbb{R}^n : Ax \leq b \}$, which is equal to $P$. Thus, by this projection argument, we conclude that the $x^*$, in the vertices $(x^*,y^*)$ of $P'$ that are created by intersecting the $2k$ new constraints with the original constraints, are equal to the vertices of $P$, which are integer. So $x^*$ is integer in all vertices $(x^*,y^*)$ of $P'$.
\end{proof}

\subsubsection{Ramping (1bin)}
\label{sec:rampsimple}
In this section, we present two different 1bin models for UC with ramping: one contains fewer constraints, the other contains additional facets, making it tighter. They are largely based on the work by Damc\i-Kurt et al.~\cite{damci-kurtPolyhedralStudyProduction2016}, who obtained 1bin facets for ramping up with start-up capabilities. We\hlorange{ }use their result\hlyellow{s} here by assuming that the start-up capabilities are the same as the regular ramping capabilities, so by substituting $\SU=\Plow+\RU$ in their constraints. After discussing the simplest valid constraints, we present their 1bin constraints that consider ramping up, and we add facets for ramping down. \hlorange{Furthermore, we present the 1bin two-period convex hull for ramping up \textit{and} down together. Our new theoretical results align with recent findings by Tian et al.~\cite{tianPolyhedralStudyUnit2026}, who claimed the same facets for ramping down and shut-down capabilities and the same convex hull formulation independently.}



Intuitively, the simplest constraint to limit the maximum ramp-up $\RU$ of a unit would the following:
\begin{align*}
    p_{t}-p_{t-1} \leq \RU,\qquad \forall t\in\mathcal{T}.
\end{align*}
However, if the minimum limit $\Plow$ is larger than the maximum ramp rate $\RU$, this constraint would imply that the unit could never \hlyellow{start up}. The simplest valid constraint, though not a facet, is:
\begin{align*}
    p_{t} - p_{t-1} \leq \RU + \Plow u_{t} - \Plow u_{t-1},\qquad \forall t\in\mathcal{T}.
\end{align*}

Damc\i-Kurt et al.~\cite{damci-kurtPolyhedralStudyProduction2016} proved that the very similar constraint~\eqref{eq:rutight} \hlorange{(see below)} is a facet of the two-period convex hull of the ramping up problem. 
Using the symmetry between ramping up and ramping down constraints, we derive the ramp-down analogue of constraint~\eqref{eq:rutight}, which is constraint~\eqref{eq:rdtight}.  
Adding constraints~\eqref{eq:rutight} and~\eqref{eq:rdtight} to \hyperref[eq:1binucmodel]{Model I} results in \hyperref[eq:1binucmodelecompact]{Model I-E}, which accurately models UC with ramping. 

\noindent\hrulefill \\
\textbf{Model I-E}: Generation limits and ramping \\
\vspace{-5pt}
\noindent\hrule
\vspace{-5pt}
\begin{subequations}
\label{eq:1binucmodelecompact}
\begin{align}
    &\eqref{eq:pboundup}-\eqref{eq:u} \notag\\
    &p_{t} - p_{t-1} \leq (\Plow + \RU)u_{t} - \Plow u_{t-1} \quad&\forall t\in \Tzero \label{eq:rutight}\\
    &p_{t-1} - p_{t} \leq (\Plow + \RD)u_{t-1} - \Plow u_{t} \quad&\forall t\in \Tzero \label{eq:rdtight}
\end{align}
\end{subequations}
\vspace{-12pt}
\noindent\hrule 
\vspace{10pt} 

However, the tightness of the formulation can be increased further. Damc\i-Kurt et al.~\cite{damci-kurtPolyhedralStudyProduction2016} proved that constraint~\eqref{eq:pbound1bintight}, which is an upper bound on $p_{t}$, is also a facet of the single-unit two-period convex hull ramping up problem.
We introduce the additional constraint~\eqref{eq:plowbound1bintight} for ramping down. 
Adding additional facets~\eqref{eq:pbound1bintight} and~\eqref{eq:plowbound1bintight} to \hyperref[eq:1binucmodelecompact]{Model I-E} results in \hyperref[eq:1binucmodele]{Model I-E-T}. \hlblue{Note that when modeling clustered }\hlteal{UC}\hlblue{, \mbox{\hyperref[eq:1binucmodele]{Model I-E-T}} should be used. The constraints from \mbox{\hyperref[eq:1binucmodelecompact]{Model I-E}} alone do not model the problem for clustered }\hlteal{UC}\hlblue{ accurately.}\footnote{Consider, for example, a situation where $u_{t-1}=1$ and one additional unit starts up in period $t$, so $u_{t}=2$. Constraint~\eqref{eq:rutight} then becomes $p_{t}-p_{t-1}\leq \Plow + 2\RU$. Generation values $p_{t-1}=\Pup$ and $p_{t}=\Pup+\Plow + 2\RU$ satisfy this constraint. However, they are not actually feasible. The unit that starts up in period $t$ can produce at most $\Plow + \RU$, so $p_{t}$ can be at most $\Pup+\Plow + \RU$. Knueven et al.~\cite{knuevenExploitingIdenticalGenerators2018} refer to this infeasible scenario as one unit ``stealing'' ramping capabilities from the other. Constraint~\eqref{eq:pbound1bintight} becomes $p_{t}\leq \Pup + \Plow + \RU$, so this constraint can ensure that the clustered \hlteal{UC} problem is accurately modeled.}

\noindent\hrulefill \\
\textbf{Model I-E-T}: Generation limits and ramping \\
\vspace{-5pt}
\noindent\hrule
\vspace{-5pt}
\begin{subequations}
\label{eq:1binucmodele}
\begin{align}
    &\eqref{eq:pboundup}-\eqref{eq:u},\eqref{eq:rutight},\eqref{eq:rdtight} \notag\\
    &p_{t} \leq (\Plow + \RU)u_{t} + (\Pup - \Plow - \RU) u_{t-1} \quad&\forall t\in \Tzero \label{eq:pbound1bintight}\\
    &p_{t-1} \leq (\Plow + \RD)u_{t-1} + (\Pup - \Plow - \RD) u_{t} &\forall t\in \Tzero \label{eq:plowbound1bintight} 
\end{align}
\end{subequations}
\vspace{-10pt}
\noindent\hrule 
\vspace{10pt} 

Damc\i-Kurt et al.~\cite{damci-kurtPolyhedralStudyProduction2016} obtained the 1bin two-period convex hull of the ramping up problem. 
\hlyellow{Moreover, Tian et \mbox{al.\cite{tianPolyhedralStudyUnit2026}} }\hlorange{recently}\hlyellow{ obtained the 1bin two-period convex hull of for ramping up \textit{and} down together}\hlorange{ (\mbox{\hyperref[eq:1binucmodele]{Model I-E-T}})}\hlorange{. Independently, we obtained the same convex hull formulation, and in Proposition~\ref{obs:rampingconvexhull} in Appendix~\ref{app:proofs1binramp}, we present different, shorter proof for this result. The proof }\hlyellow{follows from the more extensive \textit{3bin} two-period convex hull for ramping and start-up and shut-down capabilities, which we prove later in Section~\ref{sec:susdcap3bin}.} 
It follows that \hyperref[eq:1binucmodele]{Model I-E-T} is the tightest possible formulation for the single-unit two-period ramping problem.

\begin{table*}[!ht]
    \centering
    \caption{Theoretical overview of 1bin UC convex hulls 
    }
    \label{tab:convexhulls1bin}
    \begin{tabular}{lll}
    \hline
        Problem & Convex hull & Proven by \\
        \hline
        1bin generation limits (\hyperref[eq:1binucmodel]{\textbf{Model I}}) for $|\mathcal{T}|$ periods &~\eqref{eq:pboundup},\eqref{eq:pboundslow}, $u_{t}\leq 1$ & us: Lemma~\ref{thm:thelemma} \\
        1bin ramping up for $|\mathcal{T}|=2$ &~\eqref{eq:pboundup},\eqref{eq:pboundslow},\eqref{eq:rutight},\eqref{eq:pbound1bintight}, $u_{t}\leq 1$  & \cite{damci-kurtPolyhedralStudyProduction2016}\\
        1bin ramping down for $|\mathcal{T}|=2$ &~\eqref{eq:pboundup},\eqref{eq:pboundslow},\eqref{eq:rdtight},\eqref{eq:plowbound1bintight}, $u_{t}\leq 1$ & \cite{tianPolyhedralStudyUnit2026} \& us: Proposition~\ref{obs:rampingconvexhull}\\
        1bin ramping up \textit{and} down (\hyperref[eq:1binucmodele]{\textbf{Model I-E-T}}) for $|\mathcal{T}|=2$ &~\eqref{eq:pboundup},\eqref{eq:pboundslow},\eqref{eq:rutight},\eqref{eq:rdtight},\eqref{eq:pbound1bintight},\eqref{eq:plowbound1bintight}, $u_{t}\leq 1$ & \cite{tianPolyhedralStudyUnit2026} \& us: Proposition~\ref{obs:rampingconvexhull}\\
        1bin ramping up and start-up capability for $|\mathcal{T}|=2$ &~\eqref{eq:pboundup},\eqref{eq:pboundslow},\eqref{eq:rutight},\eqref{eq:su1bintight},\eqref{eq:supbound1bintight}, $u_{t}\leq 1$ & \cite{damci-kurtPolyhedralStudyProduction2016}\\
        1bin ramping down and shut-down capability for $|\mathcal{T}|=2$ &~\eqref{eq:pboundup},\eqref{eq:pboundslow},\eqref{eq:rdtight},\eqref{eq:sd1bintight},\eqref{eq:sdplowbound1bintight}, $u_{t}\leq 1$ & \cite{tianPolyhedralStudyUnit2026} \& us: Proposition~\ref{obs:susdconvexhull}\\
        1bin ramping up \textit{and} down and start-up \textit{and} shut-down &~\eqref{eq:pboundup},\eqref{eq:pboundslow},\eqref{eq:rutight},\eqref{eq:rdtight},\eqref{eq:su1bintight},\eqref{eq:sd1bintight},\eqref{eq:supbound1bintight},\eqref{eq:sdplowbound1bintight}, $u_{t}\leq 1$ & \cite{tianPolyhedralStudyUnit2026} \& us: Proposition~\ref{obs:susdconvexhull}\\
        \quad capabilities (\hyperref[eq:1binucmodele2]{\textbf{Model I-E2-T}}) for $|\mathcal{T}|=2$ & & \\
        \hline
    \end{tabular}
\end{table*}

\subsubsection{Start-up and shut-down capabilities (1bin)}
\label{sec:susdcap1bin}

The simplest constraint that ensures a power limit of $\SU$ when starting up, and respects the ramping limits, would be the following:
\begin{align*}
    p_{t}-p_{t-1} \leq \SU - (\SU-\RU)u_{t-1},\qquad \forall t\in\mathcal{T}.
\end{align*}
However, tighter formulations for this problem have been found in the literature. 
As mentioned in the previous section, Damc\i-Kurt et al.~\cite{damci-kurtPolyhedralStudyProduction2016} obtained \hlorange{the} 1bin convex hull formulation for ramping up with start-up capabilities. 
\hlorange{In this section, we present the 1bin convex hull for ramping and start-up and shut-down capabilities, which align with recent findings by Tian et al.\cite{tianPolyhedralStudyUnit2026}, who claimed the same convex hull formulation independently.}

\hlorange{More specifically, }for the two-period ramping up and start-up problem, Damc\i-Kurt et al.~\cite{damci-kurtPolyhedralStudyProduction2016} proved that constraints~\eqref{eq:rutight},~\eqref{eq:su1bintight}, and~\eqref{eq:supbound1bintight} are facets\hlorange{, and we add constraints~\eqref{eq:rdtight},~\eqref{eq:sd1bintight}, and~\eqref{eq:sdplowbound1bintight} for ramping down and shut-down}. Constraints~\eqref{eq:su1bintight} and ~\eqref{eq:sd1bintight} alone are sufficient to model the problem accurately; the others just improve the tightness of the model. Therefore, we first introduce \hyperref[eq:1binucmodele2compact]{Model I-E2}. 

\noindent\hrulefill \\
\textbf{Model I-E2}: Generation limits, ramping, and start-up and shut-down capabilities \\
\vspace{-5pt}
\noindent\hrule
\vspace{-5pt}
\begin{subequations}
\label{eq:1binucmodele2compact}
\begin{align}
    &\eqref{eq:pboundup}-\eqref{eq:u} \notag\\
    &p_{t} - p_{t-1} \leq \SU u_{t} - (\SU - \RU) u_{t-1} &\forall t\in \Tzero \label{eq:su1bintight}\\
    &p_{t-1} - p_{t} \leq \SD u_{t-1} - (\SD - \RD) u_{t} &\forall t\in \Tzero \label{eq:sd1bintight} 
\end{align}
\end{subequations}
\vspace{-10pt}
\noindent\hrule 
\vspace{10pt} 

Furthermore, we introduce \hyperref[eq:1binucmodele2]{Model I-E2-T}, which is a tighter version of \hyperref[eq:1binucmodele2compact]{Model I-E2}, that contains additional facets~\eqref{eq:rutight},~\eqref{eq:rdtight},~\eqref{eq:supbound1bintight}, and~\eqref{eq:sdplowbound1bintight}. Note \hlblue{again} that when modeling clustered \hlteal{UC}, \hyperref[eq:1binucmodele2]{Model I-E2-T} should be used. The constraints from \hyperref[eq:1binucmodele2compact]{Model I-E2} alone do not model the problem for clustered \hlteal{UC} accurately.

\noindent\hrulefill \\
\textbf{Model I-E2-T}: Generation limits, ramping, and start-up and shut-down capabilities \\
\vspace{-5pt}
\noindent\hrule
\vspace{-5pt}
\begin{subequations}
\label{eq:1binucmodele2}
\begin{align}
    &\eqref{eq:pboundup}-\eqref{eq:u},\eqref{eq:rutight},\eqref{eq:rdtight},\eqref{eq:su1bintight},\eqref{eq:sd1bintight} \hspace{-50pt} \notag\\
    &p_{t} \leq \SU u_{t} + (\Pup - \SU) u_{t-1} \quad&\forall t\in \Tzero \label{eq:supbound1bintight}\\
    &p_{t-1} \leq \SD u_{t-1} + (\Pup - \SD) u_{t} \quad&\forall t\in \Tzero \label{eq:sdplowbound1bintight} 
\end{align}
\end{subequations}
\vspace{-10pt}
\noindent\hrule 
\vspace{10pt} 



\hlyellow{Tian et al.\cite{tianPolyhedralStudyUnit2026} }\hlorange{recently}\hlyellow{ claimed that the LP relaxation of \mbox{\hyperref[eq:1binucmodele2compact]{Model I-E2}} }\hlgreen{defines}\hlyellow{ the 1bin two-period convex hull of the UC problem with ramping and start-up and shut-down capabilities.} In Proposition~\ref{obs:susdconvexhull} in Appendix~\ref{app:proofs1binramp}, \hlorange{we present a different, shorter proof for this result, }which also follows from the more extensive 3bin proof in Section~\ref{sec:susdcap3bin}.
It follows that \hyperref[eq:1binucmodele2]{Model I-E2-T} is the tightest possible model for the two-period ramping and start-up and shut-down capabilities problem.

\subsubsection{Overview proven facets and convex hulls (1bin)}
\label{sec:overview1bin}

To conclude Section~\ref{sec:1bin}, we now present an overview of the theoretical insights. Table~\ref{tab:convexhulls1bin} \hlyellow{provides} an overview of \hlyellow{proven} 1bin convex hulls\hlyellow{.} It lists all the sets of constraints \hlyellow{(facets)} that are proven to describe the convex hull of a certain 1bin UC problem, and who proved \hlyellow{them}. 

It follows from Table~\ref{tab:convexhulls1bin} that basic UC \hyperref[eq:1binucmodel]{Model I} is the tightest possible formulation for the MILP problem it models. The same is true for \hyperref[eq:1binucmodele]{Model I-E-T} and \hyperref[eq:1binucmodele2]{Model I-E2-T} for the two-period problem. \hyperref[eq:1binucmodelecompact]{Model I-E} and \hyperref[eq:1binucmodele2compact]{Model I-E2} are less tight, but their formulation is \hlyellow{much} smaller (contains fewer constraints). 

\subsection{Level II: 3bin UC Models}
\label{sec:3bin}

In this section, we consider 3bin UC models that contain the additional start-up and shut-down variables \hlgreen{in the constraint set. These variables are essential to model start-up and shut-down costs, and they enable the construction of polynomial-size convex hull formulations for models that would otherwise require exponentially many constraints.} We present new theoretical insights into the tightness of two often-used formulations for start-up and shut-down costs, as well as the 3bin convex hull for ramping and start-up and shut-down capabilities.

\subsubsection{Start-up and shut-down variables}
\label{sec:susdvariables}

We begin by considering \hyperref[eq:3binucmodel]{Model II}, introduced in Section~\ref{sec:basicucmodel}, and by explaining some general theoretical insights into 3bin models. 


As briefly mentioned in Section~\ref{sec:defininglevels}, if $u_{t}$ is defined as binary, then the values of $v_{t}$ and $w_{t}$ are completely fixed by constraints~\eqref{eq:commitment}-\eqref{eq:wcommitment}~\cite{rajanMinimumPolytopesUnit2005}. Thus, the start-up and shut-down variables can be defined as binary or as non-negative real variables; both result in the same solutions. From a solving perspective, the assumed benefit of relaxing the integrality constraints is that fewer variables need to be branched on, leading to a smaller enumeration tree~\cite{ostrowskiTightMixedInteger2012}. However, modern MIP solvers contain tools that can exploit the integrality of these variables, such as cutting planes and node presolve~\cite{ostrowskiTightMixedInteger2012,morales-espanaTightCompactMILP2013}. Therefore, defining the variables as integers \hlorange{could} result in faster computations of the optimal solution. The best strategy, according to Ostrowski et al.~\cite{ostrowskiTightMixedInteger2012}, would be to adjust the branching priority of these variables so that they are chosen only after all other integer variables have been fixed.

Furthermore, 2bin formulations are occasionally used, for example by using\hlyellow{ }projection techniques~\cite{yangNovelProjectedTwo2016}. Any of the presented 3bin models in this paper can be easily rewritten to a 2bin model by substituting $w_{t} = v_{t} - u_{t} + u_{t-1}$ from~\eqref{eq:commitment} in constraints~\eqref{eq:wcommitment} and $w_{t}\geq0$, resulting in constraints~\eqref{eq:vcommitment2bin} and~\eqref{eq:vlowerbound}. 

\begin{subequations}
\begin{align}
    &v_{t} \leq 1 - u_{t-1} \qquad&\forall g\in \mathcal{G},t\in \Tzero \label{eq:vcommitment2bin}\\
    &v_{t} \geq u_{t} - u_{t-1} \qquad&\forall g\in \mathcal{G},t\in \Tzero \label{eq:vlowerbound}
\end{align}
\end{subequations}




\subsubsection{Start-up and shut-down costs}
\label{sec:susdcosts}

In this section, we consider the costs associated with a unit starting up and shutting down. There are two different ways that basic start-up and shut-down costs are typically modeled. We explain both formulations, and show that they are equivalent with some new theoretical insights. \hlgreen{We remind the reader that we assume that all costs are nonnegative.}

It should be mentioned that we only consider one type of start-up and shut-down costs. Others, such as Morales-Espa\~{n}a et al.~\cite{morales-espanaTightCompactMILP2013}, Queyranne and Wolsey~\cite{queyranneTightMIPFormulations2017}, and Knueven et al.~\cite{knuevenNovelMatchingFormulation2020}, presented (tight) formulations \hlyellow{for an extension with} costs for start-up that may depend on the time that the generator has been idle. This is often referred to as hot, warm, and cold start-ups, but \hlyellow{tight} formulations \hlyellow{for this problem} require additional variables. 

The start-up and shut-down variables $v_{gt}$ and $w_{gt}$ can be used to model start-up and shut-down costs in the objective, as explained in Section~\ref{sec:basicucmodel}, resulting in \hlblue{the constraints of} \hyperref[eq:3binucmodel]{Model II}\hlyellow{, which we also presented previously in Section~\ref{sec:basicucmodel}}. 
We show in Theorem~\ref{obs:3binCH} that \hyperref[eq:3binucmodel]{Model II} is the tightest possible formulation of this problem.


\begin{theorem}
\label{obs:3binCH}
    Constraints~\eqref{eq:pboundup},~\eqref{eq:pboundslow},~\eqref{eq:commitment}-\eqref{eq:wcommitment}, and $v_{t},\ w_{t}\geq 0$ describe the convex hull of the problem given by constraints~\eqref{eq:pboundup}-\eqref{eq:u} and~\eqref{eq:commitment}-\eqref{eq:w}.
\end{theorem}
\begin{proof}
    It was proven by Rajan and Takriti~\cite{rajanMinimumPolytopesUnit2005} that constraints~\eqref{eq:commitment}-\eqref{eq:wcommitment} and $v_{t},\ w_{t}\geq 0$ completely describe the convex hull of the IP problem given by constraints~\eqref{eq:commitment}-\eqref{eq:w} and~\eqref{eq:u}. 
    They actually proved this for a more complex version of the model where minimum up and down times are considered, but the proof holds here if the minimum up/down time is 1. 
    When adding a continuous variable $p_{t}$, then the binary variables are still binary in all extreme points of the LP relaxation of the new MIP by Lemma~\ref{thm:thelemma}. Thus the theorem holds.
\end{proof}

As explained in Section~\ref{sec:3bin}, relaxing the integrality constraint on the start-up and shut-down variables does not change the problem. The second way that this problem is commonly modeled takes advantage of this~\cite{nowakStochasticLagrangianRelaxation2000,carrionComputationallyEfficientMixedinteger2006}. Instead of introducing the variables $v_{gt}$ and $w_{gt}$, Nowak and Roemisch~\cite{nowakStochasticLagrangianRelaxation2000} introduced an auxiliary continuous variable $\csu_{gt}$, which represents the costs associated with starting up unit $g$ in time period $t$. Similarly, we can introduce an auxiliary continuous variable $\csd_{gt}$ to represent the shutdown costs. Essentially, these variables represent the value of $\SUC_g v_{gt}$ and $\SDC_g w_{gt}$, respectively. The \hlblue{function in the} objective that minimizes production costs, non-load costs, and start-up and shut-down costs can then be written as:

\begin{align*}
    f^\mathrm{p}(p_{gt},u_{gt},v_{gt},w_{gt})= & \sum_{g\in\mathcal{G}} \sum_{t\in\mathcal{T}} \left( \Cp_g p_{gt} + \Cu_g u_{gt} + \csu_{gt} + \csd_{gt} \right). 
\end{align*}

The following constraints are included to ensure that the cost variables attain the correct values: 

\begin{subequations}
\begin{align*}
    &\csu_{gt} \geq \SUC_g(u_{gt} - u_{g,t-1}) \qquad&\forall g\in \mathcal{G},t\in \Tzero,\\
    &\csd_{gt} \geq \SDC_g(u_{g,t-1} - u_{gt}) \qquad&\forall g\in \mathcal{G},t\in \Tzero,\\
    &\csu_{gt},\ \csd_{gt} \in \mathbb{R}_{\geq0} \qquad&\forall g\in \mathcal{G},t\in \mathcal{T}.
\end{align*}
\end{subequations}

Together with constraints~\eqref{eq:pboundup}-\eqref{eq:u}, a MILP formulation is obtained that also accurately models and minimizes start-up and shut-down costs. 
To be able to compare this formulation to \hyperref[eq:3binucmodel]{Model II}, we first rewrite it such that it contains the $v_{t}$ and $w_{gt}$ variables, instead of the $\csu_{gt}$ and $\csd_{gt}$ variables. By substituting $\csu_{gt} = \SUC_g v_{gt}$ and $\csd_{gt} = \SDC_g w_{gt}$, and dropping the $g$ subscript, we obtain\hlorange{ }\hyperref[eq:3binucmodelcompact]{Model II-S}, which is almost the same as the formulation \hlorange{of} Muckstadt and Koenig~\cite{muckstadtApplicationLagrangianRelaxation1977} \hlorange{from} 1977. 


\noindent\hrulefill \\
\textbf{Model II-S}: Generation limits and start-up \& shut-down costs \\
\vspace{-5pt}
\noindent\hrule
\vspace{-5pt}
\begin{subequations}
\label{eq:3binucmodelcompact}
\begin{align}
    \quad &\eqref{eq:pboundup}-\eqref{eq:u},\eqref{eq:v},\eqref{eq:w} \hspace{-100pt}\notag\\
    &v_{t} \geq u_{t} - u_{t-1} \qquad&\forall t\in \Tzero \label{eq:vlowerbound2}\\
    &w_{t} \geq u_{t-1} - u_{t} \qquad&\forall t\in \Tzero \label{eq:wlowerbound}
\end{align}
\end{subequations}
\vspace{-10pt}
\noindent\hrule 
\vspace{10pt} 

It also holds for \hyperref[eq:3binucmodelcompact]{Model II-S} that relaxing the integrality constraint on the start-up and shut-down variables does not change the problem~\cite{muckstadtApplicationLagrangianRelaxation1977}. We observe that \hyperref[eq:3binucmodelcompact]{Model II-S} contains fewer constraints than \hyperref[eq:3binucmodel]{Model II}. We show in Theorem~\ref{obs:lbCH} that the solution space of the LP relaxations of \hyperref[eq:3binucmodel]{Model II} and \hyperref[eq:3binucmodelcompact]{Model II-S} are exactly the same in the direction of the objective function~\eqref{eq:objsusdcosts}:
\begin{align}
    \min\quad f^\mathrm{inv}(\uinv)+\sum_{t\in\mathcal{T}} \left(\Cp p_{t} + \Cu u_{t} + \SUC v_t + \SDC w_t \right).\label{eq:objsusdcosts} 
\end{align}
To this end, we first inspect \hyperref[eq:3binucmodel]{Model II} without constraints~\eqref{eq:vcommitment} and~\eqref{eq:wcommitment}, which is how Garver~\cite{garverPowerGenerationScheduling1962} originally modeled the UC problem in 1962. In Lemma~\ref{obs:noupperbounds}, we show  that upper bounds~\eqref{eq:vcommitment} and~\eqref{eq:wcommitment} in \hyperref[eq:3binucmodel]{Model II} on the start-up and shut-down variables $v_{t}$ and $w_{t}$ \hlorange{are redundant when the} objective function is~\eqref{eq:objsusdcosts}, which minimizes the start-up and shut-down costs. 

\begin{lemma}
    \hlblue{If the objective function is~\eqref{eq:objsusdcosts} and} \hlorange{$\SUC,\SDC>0$, then constraints~\eqref{eq:vcommitment} and~\eqref{eq:wcommitment} are redundant in the LP relaxation of \mbox{\hyperref[eq:3binucmodel]{Model II}}.}
    \label{obs:noupperbounds} 
\end{lemma}
\begin{proof}
    \hlorange{Let $\SUC,\SDC>0$. We define $P$ as the set of solutions to the LP relaxation of \mbox{\hyperref[eq:3binucmodel]{Model II}}, and $Q$ as the set of solutions to the LP relaxation of \mbox{\hyperref[eq:3binucmodel]{Model II}} without constraints~\eqref{eq:vcommitment} and~\eqref{eq:wcommitment}. By definition, we know that $P\subset Q$. We look into all points in $Q\setminus P$, so all solutions that satisfy the constraints of the LP relaxation of \mbox{\hyperref[eq:3binucmodel]{Model II}}, but violate constraint~\eqref{eq:vcommitment} or~\eqref{eq:wcommitment}, and show that these solutions will not be attained in the direction of the objective function~\eqref{eq:objsusdcosts}. To this end, we consider a solution $x'\in Q\setminus P$, so a solution in $Q$ that satisfies relaxed \mbox{constraints~\eqref{eq:pboundup}-\eqref{eq:u}},\eqref{eq:commitment},\eqref{eq:v}, and~\eqref{eq:w}, but violates constraint~\eqref{eq:vcommitment} or~\eqref{eq:wcommitment}. So $v_{t'}> u_{t'}$ or $w_{t'} > 1-u_{t'}$ for some $t'\in\mathcal{T}$ in the solution $x'$.}

    \hlorange{We first consider the case that} $v_{t'}> u_{t'}$, so $v_{t'} - u_{t'}>0$, then it follows from equality constraint~\eqref{eq:commitment} that $v_{t'} - u_{t'} = w_{t'} - u_{t'-1}>0$, so $w_{t'} > u_{t'-1}$. We know that $u_{t'},u_{t'-1}\geq0$ by constraint~\eqref{eq:u}, so it follows that $v_{t'}> u_{t'} \geq0$ and $w_{t'} > u_{t'-1} \geq 0$. That means that we can reduce the values of variables $v_{t'}$ and $w_{t'}$ by some positive value $\epsilon$ to obtain some other \hlgreen{solution $x''$}. This solution still satisfies relaxed constraints~\eqref{eq:v} and~\eqref{eq:w} and equality constraint~\eqref{eq:commitment}, so \hlgreen{$x''$} is a \hlorange{feasible solution to \mbox{\hyperref[eq:3binucmodel]{Model II}}}. The objective function value of \hlgreen{solution $x''$} is changed by $\SUC\epsilon + \SDC \epsilon$ compared to the objective function value of \hlgreen{solution $x'$}, which is negative if $\SUC,\SDC>0$. Thus, when constraint~\eqref{eq:vcommitment} is violated by a \hlgreen{solution $x'$}, we can always obtain a different \hlorange{solution} in $Q$ with a lower objective function value. 
    
    A similar argument holds for a \hlorange{solution} $x'\in Q$ that violates constraint~\eqref{eq:wcommitment}. Thus, \hlorange{constraints~\eqref{eq:vcommitment} and~\eqref{eq:wcommitment} are redundant in the LP relaxation of \mbox{\hyperref[eq:3binucmodel]{Model II}}}. 
\end{proof}





\hlorange{Next, we show in Theorem~\ref{obs:lbCH} that all optimal solutions to }\hlgreen{the LP relaxation of}\hlorange{ \mbox{\hyperref[eq:3binucmodelcompact]{Model II-S}}, which contains the lower bounds~\eqref{eq:vlowerbound2} and~\eqref{eq:wlowerbound}, actually satisfy the equality constraint~\eqref{eq:commitment} of \mbox{\hyperref[eq:3binucmodel]{Model II}}.}

\begin{theorem}
    \hlblue{If the objective function is~\eqref{eq:objsusdcosts} and} \hlorange{$\SUC,\SDC>0$, then all optimal solutions to the LP relaxation of \mbox{\hyperref[eq:3binucmodelcompact]{Model II-S}} satisfy the equality constraint~\eqref{eq:commitment}.}
    \label{obs:lbCH}
\end{theorem}
\begin{proof}

    \hlorange{We consider two cases: $u_t - u_{t-1}\geq0$ in an optimal solution to the LP relaxation of \mbox{\hyperref[eq:3binucmodelcompact]{Model II-S}}, or $u_t - u_{t-1}\leq0$.}

    \hlorange{Suppose $u_t - u_{t-1}\geq0$ in an optimal solution to the LP relaxation of \mbox{\hyperref[eq:3binucmodelcompact]{Model II-S}}. Then it follows from constraint~\eqref{eq:vlowerbound2} that $v_t \geq u_t - u_{t-1}\geq0$, and it follows from constraint~\eqref{eq:wlowerbound} that $w_t \geq 0 \geq u_{t-1} - u_t$. If $\SUC,\SDC>0$, then the objective of \mbox{\hyperref[eq:3binucmodelcompact]{Model II-S}} minimizes the variables $v_t$ and $w_t$, so }\hlgreen{it follows from these inequalities that}\hlorange{ $v_t=u_t-u_{t-1}$ and $w_t=0$ in the optimal solution to the relaxation. These variable values satisfy the equality constraint~\eqref{eq:commitment}: $u_t - u_{t-1} = v_t - w_t$.}

    \hlorange{A similar argument follows for the case that $u_t - u_{t-1}\leq0$ in an optimal solution to the LP relaxation of \mbox{\hyperref[eq:3binucmodelcompact]{Model II-S}}. Thus, if $\SUC,\SDC>0$, then all optimal solutions to the LP relaxation of \mbox{\hyperref[eq:3binucmodelcompact]{Model II-S}} satisfy the equality constraint~\eqref{eq:commitment}.}
\end{proof}


\hlorange{Thus, all models presented in this section model the same problem, and the set of optimal solutions to their LP relaxations is the same. }\hlgreen{Therefore}\hlorange{, their tightness in the direction of the objective function~\eqref{eq:objsusdcosts} is the same. }\hlgreen{Notice that in the proof of Theorem~\ref{obs:lbCH}, we never use the upper bound of 1 for $v_t$ and $w_t$. Thus, the LP relaxation of \mbox{\hyperref[eq:3binucmodelcompact]{Model II-S}} without these upper bounds }\hlgreen{defines}\hlorange{ an unbounded convex hull}\hlgreen{ of the problem}. Constraint $u_{t}\leq 1$ is needed in the description \hlorange{of} its convex hull, since other constraints no longer imply it\hlorange{, whereas it is not needed in the description of the convex hull of \mbox{\hyperref[eq:3binucmodel]{Model II}} (shown in Theorem~\ref{obs:3binCH})}. The difference between \hlorange{\mbox{\hyperref[eq:3binucmodel]{Model II}} (without constraints~\eqref{eq:vcommitment} and~\eqref{eq:wcommitment}) and \mbox{\hyperref[eq:3binucmodelcompact]{Model II-S}}} is that the first has equality constraint~\eqref{eq:commitment}, whereas the second has two inequalities instead. We do not know a priori how these models compare computationally.

\subsubsection{Minimum up and down times}
\label{sec:minupdowntime}

In this section, we explain how minimum up and down time constraints can best be included in UC formulations, and why we do not consider 1bin formulations for this property. It was proven by Rajan and Takriti~\cite{rajanMinimumPolytopesUnit2005} that the size of the convex hull of the UC problem including minimum up and down time constraints is linear with regard to the number of time periods, when the problem is formulated with two sets of binary variables (the commitment variable and the start-up variable), so as a 2bin model. A more recent proof by Queyranne and Wolsey~\cite{queyranneTightMIPFormulations2017} extended this result to minimum up and down times that vary across the planning horizon, as well as maximum up and down times, but this is outside our scope. When modeling minimum up and down times using only the commitment variables (1bin), Lee et al.~\cite{leeMinminpolytopes2004} proved that the convex hull is exponential in the space of the binary variables. Thus, we only consider the tight formulation for minimum up and down times found by Rajan and Takriti~\cite{rajanMinimumPolytopesUnit2005}. 

We introduce \hyperref[eq:3binucmodele]{Model II-E}, where constraint~\eqref{eq:minimumuptime} models the minimum up time $\mut$, and constraint~\eqref{eq:minimumdowntime3bin} models the minimum down time $\mdt$.\footnote{
\hlblue{When considering an investment problem, the right-hand side of constraint~\eqref{eq:minimumdowntime3bin} should be replaced by $\uinv - u_t$. We elaborate on theoretical results for this extension to investment problems in Section~\ref{sec:inv}.}} 
The objective \hlblue{usually includes start-up and shut-down costs: $f^\mathrm{p}(p_{t},u_{t},v_t,w_t)=\sum_{t\in\mathcal{T}} \left( \Cp p_t + \Cu u_{t} + \SUC_g  v_{t} + \SDC_g  w_{t} \right)$}. \hlgreen{In the remainder of this paper, we assume that all costs are nonnegative.} Both were proven to be \hlgreen{facet defining for} the multi-period minimum up/down time UC problem by Rajan and Takriti~\cite{rajanMinimumPolytopesUnit2005}. They also proved that these constraints, together with equality constraint~\eqref{eq:commitment} and $v_{t},w_{t}\geq 0$ describe the convex hull of the problem for one unit. It then follows from Lemma~\ref{thm:thelemma} that adding the $p_{t}$ variables with constraints~\eqref{eq:pboundup} and~\eqref{eq:pboundslow} still results in a convex hull, so \hyperref[eq:3binucmodele]{Model II-E} is the tightest possible formulation for this problem.

\noindent\hrulefill \\
\textbf{Model II-E}: Generation limits and minimum up \& down times \\
\vspace{-5pt}
\noindent\hrule
\vspace{-5pt}
\begin{subequations}
\label{eq:3binucmodele}
\begin{align}
    &\eqref{eq:pboundup}-\eqref{eq:u},\eqref{eq:commitment},\eqref{eq:v},\eqref{eq:w}\hspace{-100pt}\notag\\
    &\sum_{i=t-\mut+1}^t v_{i} \leq u_{t} \quad&\forall t\in \Tmut \label{eq:minimumuptime}\\
    &\sum_{i=t-\mdt+1}^t w_{i} \leq 1 - u_{t} \quad&\forall t\in \Tmdt \label{eq:minimumdowntime3bin}
\end{align}
\end{subequations}
\vspace{-5pt}
\noindent\hrule 
\vspace{10pt}

\subsubsection{Ramping and start-up and shut-down capabilities (3bin)}
\label{sec:susdcap3bin}

Intuitively, the simplest constraints that describe these capabilities would be the following, first introduced by Arroyo and Conejo~\cite{arroyoOptimalResponseThermal2000}:
\begin{align*}
    p_{t}-p_{t-1} &\leq \RU u_{t-1} + \SU v_{t}\qquad &\forall t\in\Tzero,\\
    p_{t-1}-p_{t} &\leq \RD_g u_{t} + \SD w_{t}\qquad &\forall t\in\Tzero.
\end{align*}
However, much research has been done to find tighter constraints for the problem. Ostrowski et al.~\cite{ostrowskiTightMixedInteger2012} presented several facets for the problem in three time periods. Thereafter, Damc\i-Kurt et al.~\cite{damci-kurtPolyhedralStudyProduction2016} presented facets for the problem in two time periods, which contain fewer non-zero entries than those by Ostrowski et al. They proved that~\eqref{eq:su3bintight} is a facet of the ramping up polytope, and that~\eqref{eq:sd3bintight} is a facet of the ramping down polytope. 


\noindent\hrulefill \\
\textbf{Model II-E2}: Generation limits, ramping, start-up \& shut-down costs and capabilities, and minimum up \& down times, \\
\hlgreen{where $[\cdot]^+:=\max\{\cdot,0\}$.} \\
\vspace{-5pt}
\noindent\hrule
\vspace{-5pt}
\begin{subequations}
\label{eq:3binucmodele2}
\begin{align}
    &\eqref{eq:pboundslow}-\eqref{eq:u},\eqref{eq:commitment},\eqref{eq:v},\eqref{eq:w},\eqref{eq:minimumuptime},\eqref{eq:minimumdowntime3bin} \hspace{-100pt}\notag \\
    \begin{split}
        &p_{t} - p_{t-1} \leq (\SU - \Plow - \RU)v_{t} + (\Plow + \RU)u_{t} \\
        &\hspace{55pt}- \Plow u_{t-1} \hspace{90pt} \forall t\in \Tzero 
    \end{split}\label{eq:su3bintight}\\
    \begin{split}
        &p_{t-1} - p_{t} \leq (\SD - \Plow - \RD)w_{t} + (\Plow + \RD)u_{t-1} \\
        &\hspace{55pt}- \Plow u_{t} \hspace{96pt} \forall t\in \Tzero
    \end{split}
     \label{eq:sd3bintight} \\
    \begin{split}
        &p_{t} \leq \Pup u_{t} - (\Pup-\SU)v_{t} - (\Pup-\SD)w_{t+1} \\
        &\hspace{82pt}\quad\forall g\in \mathcal{G}\setminus\mathcal{G}^1,t\in \mathcal{T}\setminus \{\tstart,\tend\}
    \end{split}
     \label{eq:supbound3bintight} \\
    \begin{split}
        &p_{t} \leq \Pup u_{t} - (\Pup-\SU)v_{t} - [\SU-\SD]^+w_{t+1} \\
        &\hspace{98pt}\quad\forall g\in \mathcal{G}^1,t\in \mathcal{T}\setminus \{\tstart,\tend\}
    \end{split}
    \label{eq:supbound3bintightmultip}\\
    \begin{split}
        &p_{t} \leq \Pup u_{t} - [\SD-\SU]^+v_{t} - (\Pup-\SD)w_{t+1} \\
        &\hspace{98pt}\quad\forall g\in \mathcal{G}^1,t\in \mathcal{T}\setminus \{\tstart,\tend\}
    \end{split}
    \label{eq:sdpbound3bintightmultip}
\end{align}
\end{subequations}
\vspace{-3pt}
\noindent\hrule 
\vspace{10pt} 

Furthermore, Morales-Espa\~{n}a et al.~\cite{morales-espanaTightCompactMILP2013} introduced additional constraints~\eqref{eq:supbound3bin} and~\eqref{eq:sdplowbound3bin}. Damc\i-Kurt et al.~\cite{damci-kurtPolyhedralStudyProduction2016} later proved that these are\hlyellow{ }facets of the \hlyellow{start-}up and \hlyellow{shut-}down \hlyellow{capability} polytope, respectively.
\begin{subequations}
\begin{align}
    &p_{t} \leq \Pup u_{t} - (\Pup-\SU)v_{t} \qquad&\forall t\in \Tzero \label{eq:supbound3bin}\\
    &p_{t-1} \leq \Pup u_{t-1} - (\Pup-\SD)w_{t} \qquad&\forall t\in \Tzero \label{eq:sdplowbound3bin}
\end{align}
\end{subequations}

Morales-Espa\~{n}a et al.~\cite{morales-espanaTightCompactMILP2013} also introduced constraint~\eqref{eq:supbound3bintight}, and Gentile et al.~\cite{gentileTightMIPFormulation2017a} later proved that it \hlgreen{defines} a facet of the 3bin single-unit \textit{multi}-period \hlyellow{start-}up and \hlyellow{shut-}down \hlyellow{capability} polytope. It \hlyellow{dominates} constraints~\eqref{eq:supbound3bin} and~\eqref{eq:sdplowbound3bin}, \hlyellow{making them} redundant for $t\in \mathcal{T}\setminus \{\tstart,\tend\}$, but it is only valid for units with a minimum up time of two or higher ($\mut \geq 2$). For units with minimum up time of one ($\mut = 1$), which we denote by $\mathcal{G}^1\subseteq \mathcal{G}$, Gentile et al.~\cite{gentileTightMIPFormulation2017a} introduced constraints~\eqref{eq:supbound3bintightmultip} and~\eqref{eq:sdpbound3bintightmultip}, and proved that they are facets of the same polytope. 
Together,~\eqref{eq:supbound3bintightmultip} and~\eqref{eq:sdpbound3bintightmultip} also make~\eqref{eq:supbound3bin} and~\eqref{eq:sdplowbound3bin} redundant for $t\in \mathcal{T}\setminus \{\tstart,\tend\}$. Furthermore, constraints~\eqref{eq:supbound3bintight}-\eqref{eq:sdpbound3bintightmultip} make constraint~\eqref{eq:pboundup} redundant. Therefore, we include these constraints in our formulation.

More facets for this problem in three or more periods have been found in the literature. Pan and Guan~\cite{panPolyhedralStudyIntegrated2016} obtained the full convex hull description of the problem in three time periods. They also obtained facet-defining inequalities for the problem in more than three time periods. Knueven et al.~\cite{knuevenMixedIntegerProgramming} further extended some of this work. However, because of the large number of additional constraints that these formulations introduce, we do not consider them in this paper. Others, such as Damc\i-Kurt et al.~\cite{damci-kurtPolyhedralStudyProduction2016} and Pan and Guan~\cite{panConvexHullsUnit2017} presented valid inequalities for the multi-period problem. However, these formulations require an exponential number of constraints with respect to the number of time periods. 

Thus, we obtain the final \hyperref[eq:3binucmodele2]{Model II-E2}.\footnote{This model can be slightly adapted for specific units. If it is only desired to model ramping, but not start-up and shut-down capabilities, then these parameters can be replaced by $\SU = \Plow + \RU$ and $\SD = \Plow + \RD$. As a result, constraints~\eqref{eq:su3bintight} and~\eqref{eq:sd3bintight} become the same as the 1bin ramping constraints~\eqref{eq:rutight} and~\eqref{eq:rdtight}. 
If it is only desired to model start-up and shut-down capabilities, but not ramping, constraints~\eqref{eq:su3bintight} and~\eqref{eq:sd3bintight} can be removed.}
Damc\i-Kurt et al.~\cite{damci-kurtPolyhedralStudyProduction2016} obtained the convex hull for the two-period ramping up and start-up capability problem, and that of the ramping down and shut-down capability problem, separately. Their results do not imply that the combination of these convex hulls \hlgreen{defines} the two-period convex hull of both problems together, but in Theorem~\ref{obs:3binsusdcap} we prove that it does.

\begin{theorem}
    \label{obs:3binsusdcap}
    Constraints~\eqref{eq:pboundslow},\eqref{eq:commitment},\eqref{eq:minimumuptime},\eqref{eq:minimumdowntime3bin},\eqref{eq:su3bintight}-\eqref{eq:sdpbound3bintightmultip}, $v_{t}$ and $w_{t}\geq 0$ describe the convex hull of \hyperref[eq:3binucmodele2]{Model II-E2} for two time periods. 
\end{theorem}
\begin{proof}
We use the following method to obtain the convex hull, which is similar to the method that Elgersma et al.~\cite{elgersmaTightMIPFormulations2026a} used to obtain a convex hull for a storage problem. First, we write disjunctive sets of constraints for each of the following cases for two periods: the unit remains off, the unit turns on, the unit turns off, and the unit remains on. We then obtain the convex hull of these disjunctive sets of constraints together, using the method presented by Balas~\cite{balasDisjunctiveProgrammingHierarchy1985}. Lastly, we obtain the convex hull in the dimension of the original formulation by eliminating the additional variables using the Fourier-Motzkin elimination procedure.

We consider all constraints of \hyperref[eq:3binucmodele2]{Model II-E2}, so~\eqref{eq:pboundslow}-\eqref{eq:u},~\eqref{eq:commitment},~\eqref{eq:v},~\eqref{eq:w},~\eqref{eq:minimumuptime},~\eqref{eq:minimumdowntime3bin},~\eqref{eq:su3bintight}-\eqref{eq:sdpbound3bintightmultip}. We consider these constraints for two time periods only: $t$ and $t-1$. Note that minimum up and down time constraints~\eqref{eq:minimumuptime} and~\eqref{eq:minimumdowntime3bin} then become~\eqref{eq:vcommitment} and~\eqref{eq:wcommitment}. Moreover, variables $v_{t-1}$ and $w_{t+1}$ are not defined, so constraints~\eqref{eq:supbound3bintight}-\eqref{eq:sdpbound3bintightmultip} become~\eqref{eq:supbound3bin} and~\eqref{eq:sdplowbound3bin}. We can consider these constraints in four different cases: when $(u_{t-1},u_{t})=(0,1)$, $(1,0)$, $(1,1)$, or \hlgreen{$(0,0)$}. These four cases are associated with $(v_{t},w_{t})=(1,0)$, $(0,1)$, and $(0,0)$ \hlgreen{for the latter two} by constraints~\eqref{eq:commitment}-\eqref{eq:w}. 
By substituting these values in the constraints above and removing the redundant constraints, we obtain the following sets of constraints. 

\noindent If $(u_{t-1},u_{t},v_{t},w_{t})=(0,1,1,0)$:
\begin{align*}
    p_{t-1}=0 &\text{ by }~\eqref{eq:p} \text{ and }~\eqref{eq:sdplowbound3bin}\\
    \Plow \leq p_{t} \leq \Psu &\text{ by }~\eqref{eq:pboundslow}\text{ and }\eqref{eq:supbound3bin}
\end{align*}
If $(u_{t-1},u_{t},v_{t},w_{t})=(1,0,0,1)$:
\begin{align*}
    \Plow \leq p_{t-1} \leq \Psd &\text{ by }~\eqref{eq:pboundslow}\text{ and }\eqref{eq:sdplowbound3bin}\\
    p_{t}=0 &\text{ by }~\eqref{eq:p} \text{ and }~\eqref{eq:supbound3bin}
\end{align*}
If $(u_{t-1},u_{t},v_{t},w_{t})=(1,1,0,0)$:
\begin{align*}
    \Plow \leq p_{t-1} \leq \Pup &\text{ by }~\eqref{eq:p}\text{ and }\eqref{eq:sdplowbound3bin}\\
    \Plow \leq p_{t} \leq \Pup &\text{ by }~\eqref{eq:p}\text{ and }\eqref{eq:supbound3bin}\\
    -\RD \leq p_{t} - p_{t-1} \leq \RU &\text{ by }~\eqref{eq:sd3bintight}\text{ and }\eqref{eq:su3bintight}\end{align*}
If $(u_{t-1},u_{t},v_{t},w_{t})=(0,0,0,0)$:
\begin{align*}
    p_{t-1}=0 &\text{ by }~\eqref{eq:p} \text{ and }~\eqref{eq:sdplowbound3bin}\\
    p_{t}=0 &\text{ by }~\eqref{eq:p} \text{ and }~\eqref{eq:supbound3bin}
\end{align*}
Using the method by Balas~\cite{balasDisjunctiveProgrammingHierarchy1985}, we can obtain the convex hull of these disjunctive sets of constraints together. We remove all variables that are equal to zero, so everything from the \hlgreen{last} case is removed. We rename the variables from the \hlgreen{first} set to \hlyellow{$(\bullet) ^1$}, the variables from the \hlgreen{second} set to \hlyellow{$(\bullet) ^2$}, and all from the \hlgreen{third} set to \hlyellow{$(\bullet) ^3$}. Furthermore, we multiply all parameters in the sets with $\delta^1$, $\delta^2$, and $\delta^3$, respectively. We also need to include the additional constraint $\delta^1 + \delta^2 + \delta^3 = 1$. We obtain the following set of constraints, which describe the convex hull of the disjunctive sets of constraints:
\begin{align*}
    \Plow\delta^1 &\leq p_{t}^1 \leq \Psu\delta^1\\
    \Plow\delta^2 &\leq p_{t-1}^2 \leq \Psd\delta^2\\
    \Plow\delta^3 &\leq p_{t-1}^3 \leq \Pup\delta^3\\
    \Plow\delta^3 &\leq p_{t}^3 \leq \Pup\delta^3\\
    -\RD\delta^3 &\leq p_{t}^3 - p_{t-1}^3 \leq \RU\delta^3\\
    \delta^1 &+ \delta^2 + \delta^3 = 1
\end{align*}
where $p_{t} = p_{t}^2 + p_{t}^3$ and $p_{t-1} = p_{t-1}^1 + p_{t-1}^3$. We can rewrite this problem, such that its notation looks more similar to the original formulation. We rename $\delta^1=v_t$ and $\delta^2=w_{t}$, and we introduce $u_{t}=\delta^3+v_{t}$ and $u_{t-1}=\delta^3+w_t$, so $\delta^3=u_{t}-v_{t}=u_{t-1}-w_{t}\Rightarrow u_{t}-u_{t-1}=v_{t}-w_{t}$. Furthermore, we rename $p_{t}^3=p_t'$ and $p_{t-1}^3=p_{t-1}'$ for clarity, so $p_{t}^2 = p_{t}-p_{t}'$ and $p_{t-1}^1 = p_{t-1}-p_{t-1}'$. Rewriting the constraints above in this manner gives us:
\begin{subequations}
\begin{align}
    \Plow v_{t} &\leq p_{t}-p_{t}' \leq \Psu v_{t} \label{eq:fm1}\\
    \Plow w_{t} &\leq p_{t-1}-p_{t-1}' \leq \Psd w_{t} \label{eq:fm2}\\
    \Plow (u_{t-1}-w_{t}) &\leq p_{t-1}' \leq \Pup (u_{t-1}-w_{t}) \label{eq:fm3}\\
    \Plow (u_{t}-v_{t}) &\leq p_{t}' \leq \Pup (u_{t}-v_{t}) \label{eq:fm4}\\
    -\RD (u_{t-1}-w_{t}) &\leq p_{t}' - p_{t-1}' \leq \RU (u_{t}-v_{t}) \label{eq:fm5}\\
    w_{t} &\leq 1 - u_{t} \label{eq:fm6}\\
    u_{t}-u_{t-1} &=v_{t}-w_{t} \label{eq:fmcommitment}
\end{align}
\end{subequations}
To reduce the dimensionality of the problem, we use Fourier-Motzkin elimination to remove variables $p_{t}'$ and $p_{t-1}'$. The upper and lower bounds on $p_{t}'$ are~\eqref{eq:fm1},~\eqref{eq:fm4}, and~\eqref{eq:fm5}. Combining these according to the Fourier-Motzkin procedure and removing the redundant constraints gives us the following constraints:
\begin{subequations}
\begin{align}
    v_{t} &\geq0 \label{eq:fmfinal1}\\
    v_{t} &\leq u_{t} \label{eq:fmfinal2}\\
    p_{t} &\geq \Plow u_{t} \label{eq:fmfinal3}\\
    p_{t} &\leq \Pup u_{t} - (\Pup - \Psu)v_{t} \label{eq:fmfinal4}\\
    p_{t} - p_{t-1}' &\geq \Plow v_{t} - \RD(u_{t-1}-w_{t}) \label{eq:fm8}\\
    p_{t} - p_{t-1}' &\leq \Psu v_{t} + \RU(u_{t}-v_{t}) \label{eq:fm7}
\end{align}
\end{subequations}
The lower bounds on $p_{t-1}'$ are~\eqref{eq:fm2},~\eqref{eq:fm3}, and~\eqref{eq:fm7}. The upper bounds are~\eqref{eq:fm2},~\eqref{eq:fm3}, and~\eqref{eq:fm8}. Combining these and removing the redundant constraints gives us the following constraints:
\begin{subequations}
\begin{align}
    w_{t} &\geq 0 \label{eq:fmfinal5}\\
    p_{t-1} &\geq \Plow u_{t-1} \label{eq:fmfinal6}\\
    p_{t-1} &\leq \Pup u_{t-1} - (\Pup - \Psd) w_{t} \label{eq:fmfinal7}\\
    p_{t} - p_{t-1} &\leq (\Psu - \Plow - \RU)v_{t} + (\Plow + \RU)u_{t} - \Plow u_{t-1} \label{eq:fmfinal8}\\
    p_{t-1} - p_{t} &\leq (\Psd - \Plow - \RD)w_{t} + (\Plow + \RD)u_{t-1} - \Plow u_{t} \label{eq:fmfinal9}
\end{align}
\end{subequations}
Since constraints~\eqref{eq:fm1}-\eqref{eq:fmcommitment} describe the convex hull of the sets of disjunctive constraints in a higher dimension, and its projection is also a convex hull, we find that constraints~\eqref{eq:fm6},\eqref{eq:fmcommitment},\eqref{eq:fmfinal1}-\eqref{eq:fmfinal4},\eqref{eq:fmfinal5}-\eqref{eq:fmfinal9} describe the convex hull of the sets of disjunctive constraints in the original dimension. 

\begin{table*}[h]
    \caption{Theoretical overview of 3bin UC convex hulls}
    \label{tab:convexhulls}
    \centering
    \begin{tabular}{lll}
    \hline
        Problem & Convex hull & Proven by \\
        \hline
        3bin commitment for $|\mathcal{T}|$ periods &\eqref{eq:commitment}-\eqref{eq:wcommitment}, $v_{t}, w_{t}\geq 0$ &~\cite{rajanMinimumPolytopesUnit2005}\\
        3bin UC (\hyperref[eq:3binucmodel]{\textbf{Model II}}) for $|\mathcal{T}|$ periods &\eqref{eq:pboundup},\eqref{eq:pboundslow},\eqref{eq:commitment}-\eqref{eq:wcommitment}, $v_{t}, w_{t}\geq 0$ & us: Theorem~\ref{obs:3binCH}/\cite{gentileTightMIPFormulation2017a}?\\
        3bin UC minimizing start-up/shut-down costs &\eqref{eq:pboundup},\eqref{eq:pboundslow},~\eqref{eq:commitment}, $u_{t}\leq 1, v_{t}, w_{t}\geq 0$ & us: Lemma~\ref{obs:noupperbounds}\\
        \quad (\hyperref[eq:3binucmodel]{\textbf{Model II}} without~\eqref{eq:vcommitment},\eqref{eq:wcommitment}) for $|\mathcal{T}|$ periods & & \\
        3bin UC minimizing start-up/shut-down costs &\eqref{eq:pboundup},\eqref{eq:pboundslow},~\eqref{eq:vlowerbound2},\eqref{eq:wlowerbound}, $u_{t}\leq 1$, $v_{t}, w_{t}\geq 0$ & us: Theorem~\ref{obs:lbCH}\\
        \quad (\hyperref[eq:3binucmodelcompact]{\textbf{Model II-S}}) for $|\mathcal{T}|$ periods & & \\
        3bin minimum up and down time for $|\mathcal{T}|$ periods &\eqref{eq:commitment},\eqref{eq:minimumuptime},\eqref{eq:minimumdowntime3bin}, $v_{t}, w_{t}\geq 0$ & \cite{rajanMinimumPolytopesUnit2005} \\
        3bin minimum up and down time (\hyperref[eq:3binucmodele]{\textbf{Model II-E}}) for $|\mathcal{T}|$ periods &\eqref{eq:pboundup},\eqref{eq:pboundslow},\eqref{eq:commitment},\eqref{eq:minimumuptime},\eqref{eq:minimumdowntime3bin}, $v_{t}, w_{t}\geq 0$ & us: Lemma~\ref{thm:thelemma} \\
        3bin start-up and shut-down capabilities for $|\mathcal{T}|$ periods &\eqref{eq:pboundslow},\eqref{eq:commitment},\eqref{eq:minimumuptime},\eqref{eq:minimumdowntime3bin},\eqref{eq:supbound3bintight}-\eqref{eq:sdpbound3bintightmultip}, $v_{t}, w_{t}\geq 0$ & \cite{gentileTightMIPFormulation2017a}\\
        3bin ramping up and start-up capability for $|\mathcal{T}|=2$ &\eqref{eq:pboundslow},\eqref{eq:commitment}-\eqref{eq:wcommitment},\eqref{eq:su3bintight},\eqref{eq:supbound3bin}, $v_{t}, w_{t}\geq 0$ & \cite{damci-kurtPolyhedralStudyProduction2016}\\
        3bin ramping down and shut-down capability for $|\mathcal{T}|=2$ &\eqref{eq:pboundslow},\eqref{eq:commitment}-\eqref{eq:wcommitment},\eqref{eq:sd3bintight},\eqref{eq:sdplowbound3bin}, $v_{t}, w_{t}\geq 0$ & \cite{damci-kurtPolyhedralStudyProduction2016}\\
        3bin ramping up \textit{and} down and start-up \textit{and} shut-down &\eqref{eq:pboundslow},\eqref{eq:commitment}-\eqref{eq:wcommitment},\eqref{eq:su3bintight},\eqref{eq:sd3bintight},\eqref{eq:supbound3bin},\eqref{eq:sdplowbound3bin}, $v_{t}, w_{t}\geq 0$ & us: Theorem~\ref{obs:3binsusdcap}\\
        \quad capabilities (simplification of \hyperref[eq:3binucmodele2]{\textbf{Model II-E2}}) for $|\mathcal{T}|=2$ & & \\
        3bin ramping and start-up \textit{and} shut-down &\eqref{eq:pboundslow},\eqref{eq:commitment},\eqref{eq:minimumuptime},\eqref{eq:minimumdowntime3bin},\eqref{eq:su3bintight}-\eqref{eq:sdpbound3bintightmultip}, $v_{t}, w_{t}\geq 0$ & us: Theorem~\ref{obs:3binsusdcap}\\
        \quad capabilities (\hyperref[eq:3binucmodele2]{\textbf{Model II-E2}}) for $|\mathcal{T}|=2$ & & \\
        3bin start-up and shut-down trajectory (\hyperref[eq:3binucmodele3]{\textbf{Model II-E3}})  for $|\mathcal{T}|=2$ &\eqref{eq:pboundslow},\eqref{eq:commitment},\eqref{eq:minimumuptime},\eqref{eq:minimumdowntime3bin},\eqref{eq:su3bintight}-\eqref{eq:sdpbound3bintightmultip},\eqref{eq:susdtraj}, $v_{t}, w_{t}\geq 0$ &\cite{morales-espanaTightMIPFormulations2015a}\\
        \hline
    \end{tabular}
\end{table*}

Furthermore, constraints~\eqref{eq:vcommitment} and~\eqref{eq:wcommitment} are the same as constraints~\eqref{eq:fm6} and~\eqref{eq:fmfinal2}, respectively, which are dominated by~\eqref{eq:minimumuptime} and~\eqref{eq:minimumdowntime3bin}. Constraints~\eqref{eq:commitment} is the same as~\eqref{eq:fmcommitment}. Constraints~\eqref{eq:fmfinal3} and~\eqref{eq:fmfinal6} are the same as~\eqref{eq:pboundslow}.
Constraints~\eqref{eq:supbound3bin} and~\eqref{eq:sdplowbound3bin} are the same as~\eqref{eq:fmfinal4} and~\eqref{eq:fmfinal7}, which are dominated by~\eqref{eq:supbound3bintight}-\eqref{eq:sdpbound3bintightmultip}. Lastly, constraints~\eqref{eq:su3bintight} and~\eqref{eq:sd3bintight} are the same as~\eqref{eq:fmfinal9} and~\eqref{eq:fmfinal8}.
In conclusion, constraints~\eqref{eq:pboundslow},\eqref{eq:commitment},\eqref{eq:minimumuptime},\eqref{eq:minimumdowntime3bin},\eqref{eq:su3bintight}-\eqref{eq:sdpbound3bintightmultip}, $v_{t}$ and $w_{t}\geq 0$ describe the convex hull of \hyperref[eq:3binucmodele2]{Model II-E2} for two time periods. We have checked all combinations, and none of the constraints are redundant, so they are all facets.
\end{proof}

It follows from Theorem~\ref{obs:3binsusdcap} that \hyperref[eq:3binucmodele2]{Model II-E2} is the tightest formulation of the problem for two periods. Moreover, some of the two-period facets are replaced by dominant three-period facets in \hyperref[eq:3binucmodele2]{Model II-E2}. 


\subsubsection{Start-up and shut-down trajectories}
\label{sec:susdtrajectories}
\hlyellow{Slow startup units take more than one period to reach their minimum output. As explained in Section~\ref{sec:notation}, they produce a predefined trajectory, which is often ignored in UC formulations\cite{morales-espanaTightCompactMILP2013a}.}
Morales-Espa\~na et al.~\cite{morales-espanaTightCompactMILP2013a} presented a formulation to model start-up and shut-down trajectories. For readability (though not strictly necessary), \hlyellow{additional variables $p^\mathrm{traj}_t$ and $p^\mathrm{T}_t$ are introduced, which represent the additional output during these trajectories, and the total power output of a generator, respectively.}
\hlyellow{Their values are}\hlgreen{ determined by }\hlyellow{constraints~\eqref{eq:ptraj} and~\eqref{eq:susdtraj}, and the new variable $p^\mathrm{T}_t$ is included }\hlgreen{in }\hlyellow{the energy balance constraint~\eqref{eq:balance} instead of $p_t$.} \hlyellow{ }

\noindent\hrulefill \\
\textbf{Model II-E3}: Generation limits, ramping, minimum up \& down times, start-up \& shut-down costs, capabilities, and trajectories  \\
\vspace{-5pt}
\noindent\hrule
\vspace{-5pt}
\begin{subequations}
\label{eq:3binucmodele3}
\begin{align}
    &\eqref{eq:pboundslow}-\eqref{eq:u},\eqref{eq:commitment},\eqref{eq:v},\eqref{eq:w},\eqref{eq:minimumuptime},\eqref{eq:minimumdowntime3bin},\eqref{eq:su3bintight}-\eqref{eq:sdpbound3bintightmultip} \notag \\
    &p^\mathrm{traj}_{t} = \sum_{i=1}^{\sut} \Psu_{i}v_{\sut+t+1-i} + \sum_{i}^{\sdt} \Psd_{gi}w_{t-i+1} \label{eq:ptraj}\\
    &p^\mathrm{T}_t = p_t + p^\mathrm{traj}_{t} \label{eq:susdtraj}
\end{align}
\end{subequations}
\vspace{-5pt}
\noindent\hrule 
\vspace{10pt} 

\hlyellow{Thus, \mbox{\hyperref[eq:3binucmodele3]{Model II-E3}} is obtained by adding variables $p^\mathrm{traj}_t$ and $p^\mathrm{T}_t$ and constraints~\eqref{eq:ptraj} and~\eqref{eq:susdtraj} to previous \mbox{\hyperref[eq:3binucmodele2]{Model II-E2}}. It follows directly from Lemma~\ref{thm:thelemma} that the LP-relaxation of \mbox{\hyperref[eq:3binucmodele3]{Model II-E3}} }\hlgreen{defines}\hlyellow{ the two-period convex hull of its MIP problem. This was also proven by Morales-Espa\~na et al.\cite{morales-espanaTightMIPFormulations2015a}.}

\subsubsection{Overview proven facets and convex hulls (3bin)}
\label{sec:overview3bin}

To conclude Section~\ref{sec:3bin}, we give an overview in Table~\ref{tab:convexhulls} of the relevant 3bin convex hulls that have been found. 
It follows from this table that UC \hyperref[eq:3binucmodel]{Model II} and \hyperref[eq:3binucmodele]{Model II-E} are the tightest possible formulations of their respective problems. \hyperref[eq:3binucmodelcompact]{Model II-S} contains fewer constraints and \hlgreen{defines} an unbounded convex hull, but it is equally tight as \hyperref[eq:3binucmodel]{Model II} in the direction of their objective. The LP relaxation of \hyperref[eq:3binucmodele2]{Model II-E2} and \hyperref[eq:3binucmodele3]{Model II-E3} are the tightest possible formulations for their respective MILP problems in two periods. 

\subsection{Extension to investment UC problems}
\label{sec:inv}

\hlgreen{Every model discussed in this paper can be }\hlteal{extended}\hlgreen{ to investment problems, by simply adding the investment variable $\uinv_g$ and constraints~\eqref{eq:inv} and~\eqref{eq:uinv}}\hlblue{, and changing the right-hand side of constraint~\eqref{eq:wcommitment} or~\eqref{eq:minimumdowntime3bin} (if included) to $\uinv - u_{gt}$.}
\hlgreen{In Lemma~\ref{obs:thminvestment}, we prove that any of the insights into tightness of the models in this paper still hold when adding the investment variable and these constraints}.

\begin{lemma}
\hlinv{Consider any of the proven convex hulls from Tables~\ref{tab:convexhulls1bin} and~\ref{tab:convexhulls}, and denote the polytope that they describe by $P$. Observe that they all contain exactly one of the following constraints: $u_t\leq1$,~\eqref{eq:wcommitment}, or~\eqref{eq:minimumdowntime3bin}. We refer to these constraints as the maximum commitment (MC) constraints. The right-hand side of all MC constraints is $1$:}
\begin{align}
    u_t &\leq 1 \qquad\forall t\in\mathcal{T} \nonumber\\
    w_t + u_t &\leq 1 \qquad\forall t\in\Tzero\tag{\ref{eq:wcommitment}} \\
    \sum_{i=t-\mdt+1}^t w_{i} + u_t &\leq 1 \qquad\forall t\in\Tmdt
    \tag{\ref{eq:minimumdowntime3bin}}
\end{align}
\hlinv{We now introduce the new variable $\uinv$, replace the right-hand side of the MC constraints by $\uinv$, and add the constraint $\uinv\leq 1$, such that we obtain the new polytope $P'$. It then follows that all variables of $P'$, except $p_t\ \forall t\in\mathcal{T}$, are integer in all extreme points of $P'$. }
\label{obs:thminvestment}
\end{lemma}
\begin{proof}
    \hlinv{We denote the variables of $P$ that we are interested in, so all but $p_t\ \forall t\in\mathcal{T}$, by $x$. We use a similar proof structure to Lemma~\ref{thm:thelemma}, to prove that $x$ and $\uinv$ are integer in all vertices of $P'$. First, we consider the intersections between the new constraints, so $\uinv \leq 1$ and all MC constraints with the replaced right-hand sides. All changed MC constraints intersect in the origin, so $x$ and $\uinv$ are integer in this vertex. Furthermore, all changed MC constraints intersect $\uinv\leq1$. In the vertices of these intersections, it holds that $\uinv=1$, so the left-hand side of the MC constraints equals 1 in these vertices. Thus, all vertices satisfy the original MC constraints with equality. These constraints are facets of $P$. Therefore, all variables $x$ of $P$ are integer in the new vertices, and $\uinv=1$ so it is also integer. 
    
    Second, we consider the intersection between the new constraints ($\uinv \leq1$ and the changed MC constraints) and the original constraints of $P$, and verify that $x$ and $\uinv$ are integer in these vertices. To this end, we obtain the projection of $P'$ onto the variables of $P$, by performing Fourier-Motzkin elimination on the new $\uinv$ variable. The lower bounds on the $\uinv$ variable are the changed MC constraints, and the upper bound is $\uinv\leq1$. Combining all upper and lower bounds on $\uinv$ to eliminate the variable gives us the original MC constraints, so the projected polytope is equal to $P$. Thus, by this projection argument, we conclude that $x$ are integer in the vertices of $P'$. Moreover, we observe that we have either $\uinv=1$ or $\uinv$ is equal to the left-hand side of the MC constraints in the vertices. In the first case, $\uinv$ is obviously integer. In the second case, $\uinv$ is also integer, since the left-hand side of the MC constraints is a summation of (some of) the $x$-variables of $P$. Since these variables are integer in the vertices, their sum is also integer, so $\uinv$ is integer.
    So $x$ and $\uinv$ are integer in all vertices of $P'$.}
\end{proof}



\section{Conclusion}
\hlblue{In this paper, we presented several 1bin and 3bin }\hlteal{unit commitment (UC)}\hlblue{ formulations, with different levels of detail as summarized in Figure~\ref{fig:levelsofdetail}. For each of the models, we presented theoretical insights into the tightness of the formulations, as summarized in Tables~\ref{tab:convexhulls1bin} and~\ref{tab:convexhulls}. In conclusion, \mbox{\hyperref[eq:1binucmodel]{Model I}}, \mbox{\hyperref[eq:3binucmodel]{Model II}}, and \mbox{\hyperref[eq:3binucmodele]{Model II-E}} are the tightest possible formulations of their respective problems for a single unit. Their LP relaxations describe the convex hull of the feasible space of the single-unit problems. Furthermore, the LP relaxations of \mbox{\hyperref[eq:1binucmodele]{Model I-E-T}}, \mbox{\hyperref[eq:1binucmodele2]{Model I-E2-T}}, and \mbox{\hyperref[eq:3binucmodele2]{Model II-E2}} describe the convex hull of their respective MILP problems for one unit and two periods. We showed that these results also hold when extending the formulations to 
investment problems.}

\hlblue{We also presented some alternatives for the models above, namely \mbox{\hyperref[eq:1binucmodelecompact]{Model I-E}}, \mbox{\hyperref[eq:1binucmodele2compact]{Model I-E2}}, and \mbox{\hyperref[eq:3binucmodelcompact]{Model II-S}}. Note that the first two are only valid for single-unit 
problems. These models are less tight, but contain fewer constraints than their respective tighter versions. We also explained how each 3bin model can be reformulated as a 2bin model. Which of the formulations performs best computationally for each of the respective UC problems remains an open question, that we investigate in Part II.}

\hlblue{The formulations with different levels of detail, offering a good balance between tightness and size, can improve the accuracy and computational performance of large-scale investment and operational models.}


\appendix{}

\subsection{Proofs 1bin convex hull ramping and start-up an shut-down capabilities}
\label{app:proofs1binramp}

\begin{proposition}
    \label{obs:susdconvexhull}
    Constraints~\eqref{eq:pboundup}, \eqref{eq:pboundslow}, \eqref{eq:rutight}, \eqref{eq:rdtight}, \eqref{eq:su1bintight}, \eqref{eq:sd1bintight}, \eqref{eq:supbound1bintight}, \eqref{eq:sdplowbound1bintight}, and $u_{t}\leq 1$ of \hyperref[eq:1binucmodele2]{Model I-E2-T} \hlgreen{are necessary and sufficient to} describe the convex hull of the the single-unit two-period ramping up and down and start-up and shut-down problem defined by \hyperref[eq:1binucmodele2compact]{Model I-E2}.
\end{proposition}

\begin{proof}
    This follows from the 3bin single-unit two-period ramping and start-up and shut-down capabilities convex hull proven in Theorem~\ref{obs:3binsusdcap}. By substituting $w_t=v_t - u_t + u_{t-1}$, we obtain the 2bin convex hull of the single-unit two-period problem. By performing Fourier-Motzkin elimination to remove variable $v_t$, we obtain the 1bin convex hull of the single-unit two-period problem, which consists of constraints~\eqref{eq:pboundup},~\eqref{eq:pboundslow},~\eqref{eq:rutight},~\eqref{eq:rdtight},~\eqref{eq:su1bintight},~\eqref{eq:sd1bintight},~\eqref{eq:supbound1bintight},~\eqref{eq:sdplowbound1bintight}, and $u_{t}\leq 1$. \hlgreen{None of these constraints are redundant, so they are all facet defining.} 
\end{proof} 

\begin{proposition}
    \label{obs:rampingconvexhull}
    Constraints~\eqref{eq:pboundup}, \eqref{eq:pboundslow}, \eqref{eq:rutight}, \eqref{eq:rdtight}, \eqref{eq:pbound1bintight}, \eqref{eq:plowbound1bintight}, and $u_{t}\leq 1$ of \hyperref[eq:1binucmodele]{Model I-E-T} \hlgreen{are necessary and sufficient to} describe the convex hull of the single-unit two-period ramping up and down problem defined by \hyperref[eq:1binucmodelecompact]{Model I-E}. 
\end{proposition}
\begin{proof}
    This follows from the 1bin single-unit two-period ramping and start-up and shut-down capabilities convex hull proven in Proposition~\ref{obs:susdconvexhull}. By substituting $\SU=\Plow+\RU$ and $\SD = \Plow+\RD$, we obtain the 1bin convex hull of the single-unit two-period ramping up and down problem, which consists of constraints~\eqref{eq:pboundup},\eqref{eq:pboundslow},\eqref{eq:rutight},\eqref{eq:rdtight},\eqref{eq:pbound1bintight},\eqref{eq:plowbound1bintight} and $u_{t}\leq 1$. \hlgreen{None of these constraints are redundant, so they are all facet defining.} 
\end{proof}

{\scriptsize{}\bibliographystyle{IEEEtran}
\bibliography{Medium_level_detail}

@article{arroyoOptimalResponseThermal2000,
  title = {Optimal Response of a Thermal Unit to an Electricity Spot Market},
  author = {Arroyo, J.M. and Conejo, A.J.},
  year = 2000,
  month = aug,
  journal = {IEEE Transactions on Power Systems},
  volume = {15},
  number = {3},
  pages = {1098--1104},
  issn = {1558-0679},
  doi = {10.1109/59.871739},
  urldate = {2024-11-18}
}

@article{balasDisjunctiveProgrammingHierarchy1985,
  title = {Disjunctive {{Programming}} and a {{Hierarchy}} of {{Relaxations}} for {{Discrete Optimization Problems}}},
  author = {Balas, Egon},
  year = 1985,
  month = aug,
  journal = {SIAM. J. on Algebraic and Discrete Methods},
  volume = {6},
  number = {3},
  pages = {466--486},
  doi = {10.1137/0606047},
  urldate = {2023-06-27},
  langid = {english}
}

@article{carrionComputationallyEfficientMixedinteger2006,
  title = {A Computationally Efficient Mixed-Integer Linear Formulation for the Thermal Unit Commitment Problem},
  author = {Carrion, M. and Arroyo, J.M.},
  year = 2006,
  month = aug,
  journal = {IEEE Transactions on Power Systems},
  volume = {21},
  number = {3},
  pages = {1371--1378},
  issn = {1558-0679},
  doi = {10.1109/TPWRS.2006.876672},
  urldate = {2023-12-18}
}

@article{damci-kurtPolyhedralStudyProduction2016,
  title = {A Polyhedral Study of Production Ramping},
  author = {{Damc{\i}-Kurt}, Pelin and K{\"u}{\c c}{\"u}kyavuz, Simge and Rajan, Deepak and Atamt{\"u}rk, Alper},
  year = 2016,
  month = jul,
  journal = {Mathematical Programming},
  volume = {158},
  number = {1},
  pages = {175--205},
  issn = {1436-4646},
  doi = {10.1007/s10107-015-0919-9},
  urldate = {2023-10-17},
  langid = {english}
}

@article{elgersmaTightMIPFormulations2026a,
	author = {Elgersma, Maaike B. and Morales-Espa{\~n}a, Germ{\'a}n and Aardal, Karen I. and Helist{\"o}, Niina and Kiviluoma, Juha and de Weerdt, Mathijs M.},
	doi = {10.1109/TPWRS.2026.3669407},
	issn = {1558-0679},
	journal = {IEEE Transactions on Power Systems},
	month = jul,
	number = {4},
	pages = {2428--2440},
	title = {Tight {MIP} {Formulations} for {Optimal} {Operation} and {Investment} of {Storage} {Including} {Reserves}},
	url = {https://ieeexplore-ieee-org.tudelft.idm.oclc.org/document/11419843},
	urldate = {2026-07-02},
	volume = {41},
	year = {2026}}

@article{garverPowerGenerationScheduling1962,
  title = {Power {{Generation Scheduling}} by {{Integer Programming-Development}} of {{Theory}}},
  author = {Garver, L. L.},
  year = 1962,
  month = apr,
  journal = {Transactions of the American Institute of Electrical Engineers. Part III: Power Apparatus and Systems},
  volume = {81},
  number = {3},
  pages = {730--734},
  issn = {2379-6766},
  doi = {10.1109/AIEEPAS.1962.4501405},
  urldate = {2024-10-04}
}

@article{gentileTightMIPFormulation2017a,
  title = {A Tight {{MIP}} Formulation of the Unit Commitment Problem with Start-up and Shut-down Constraints},
  author = {Gentile, C. and {Morales-Espa{\~n}a}, G. and Ramos, A.},
  year = 2017,
  month = mar,
  journal = {EURO Journal on Computational Optimization},
  volume = {5},
  number = {1},
  pages = {177--201},
  issn = {2192-4414},
  doi = {10.1007/s13675-016-0066-y},
  urldate = {2023-11-23},
  langid = {english}
}

@article{gongJointUnitCommitment2024,
  title = {Joint Unit Commitment Model for Hydro and Hydrogen Power to Adapt to Large-Scale Photovoltaic Power},
  author = {Gong, Yu and Liu, Tingxi and Liu, Pan and Tong, Xin},
  year = 2024,
  month = oct,
  journal = {Energy Conversion and Management},
  volume = {317},
  pages = {118794},
  issn = {0196-8904},
  doi = {10.1016/j.enconman.2024.118794},
  urldate = {2025-07-09}
}

@article{helistoImpactOperationalDetails2021,
  title = {Impact of Operational Details and Temporal Representations on Investment Planning in Energy Systems Dominated by Wind and Solar},
  author = {Helist{\"o}, N. and Kiviluoma, J. and {Morales-Espa{\~n}a}, G. and O'Dwyer, C.},
  year = 2021,
  journal = {Applied Energy},
  volume = {290},
  issn = {0306-2619},
  doi = {10.1016/j.apenergy.2021.116712},
  langid = {english}
}

@article{huaRepresentingOperationalFlexibility2018a,
  title = {Representing {{Operational Flexibility}} in {{Generation Expansion Planning Through Convex Relaxation}} of {{Unit Commitment}}},
  author = {Hua, Bowen and Baldick, Ross and Wang, Jianhui},
  year = 2018,
  month = mar,
  journal = {IEEE Transactions on Power Systems},
  volume = {33},
  number = {2},
  pages = {2272--2281},
  issn = {1558-0679},
  doi = {10.1109/TPWRS.2017.2735026},
  urldate = {2025-06-30}
}

@inproceedings{jianComparisonBlockBidding2002,
  title = {Comparison of Block Bidding and Hourly Bidding Based on Case Study},
  booktitle = {Proceedings. {{International Conference}} on {{Power System Technology}}},
  author = {Jian, Geng and Xifan, Wang and Xingzhong, Bai and Haoyong, Chen},
  year = 2002,
  month = oct,
  volume = {3},
  pages = {1387-1391 vol.3},
  doi = {10.1109/ICPST.2002.1067757},
  urldate = {2025-07-09}
}

@article{knuevenExploitingIdenticalGenerators2018,
  title = {Exploiting {{Identical Generators}} in {{Unit Commitment}}},
  author = {Knueven, Ben and Ostrowski, Jim and Watson, Jean-Paul},
  year = 2018,
  month = jul,
  journal = {IEEE Transactions on Power Systems},
  volume = {33},
  number = {4},
  pages = {4496--4507},
  issn = {0885-8950, 1558-0679},
  doi = {10.1109/TPWRS.2017.2783850},
  urldate = {2025-08-04},
  copyright = {https://ieeexplore.ieee.org/Xplorehelp/downloads/license-information/IEEE.html},
  langid = {english}
}

@article{knuevenMixedIntegerProgramming,
  title = {On {{Mixed Integer Programming Formulations}} for the {{Unit Commitment Problem}}},
  author = {Knueven, Bernard and Ostrowski, James and Watson, Jean-Paul},
  langid = {english}
}

@article{knuevenNovelMatchingFormulation2020,
  title = {A Novel Matching Formulation for Startup Costs in Unit Commitment},
  author = {Knueven, Bernard and Ostrowski, James and Watson, Jean-Paul},
  year = 2020,
  month = jun,
  journal = {Mathematical Programming Computation},
  volume = {12},
  number = {2},
  pages = {225--248},
  issn = {1867-2949, 1867-2957},
  doi = {10.1007/s12532-020-00176-5},
  urldate = {2025-08-04},
  langid = {english}
}

@article{leeMinminpolytopes2004,
  title = {Min-up/Min-down Polytopes},
  author = {Lee, Jon and Leung, Janny and Margot, Fran{\c c}ois},
  year = 2004,
  month = jun,
  journal = {Discrete Optimization},
  volume = {1},
  number = {1},
  pages = {77--85},
  issn = {1572-5286},
  doi = {10.1016/j.disopt.2003.12.001},
  urldate = {2024-10-04}
}

@article{meusApplicabilityClusteredUnit2018,
  title = {Applicability of a {{Clustered Unit Commitment Model}} in {{Power System Modeling}}},
  author = {Meus, Jelle and Poncelet, Kris and Delarue, Erik},
  year = 2018,
  month = mar,
  journal = {IEEE Transactions on Power Systems},
  volume = {33},
  number = {2},
  pages = {2195--2204},
  issn = {0885-8950, 1558-0679},
  doi = {10.1109/TPWRS.2017.2736441},
  urldate = {2023-05-16},
  langid = {english}
}

@article{monteroReviewUnitCommitment2022,
  title = {A {{Review}} on the {{Unit Commitment Problem}}: {{Approaches}}, {{Techniques}}, and {{Resolution Methods}}},
  shorttitle = {A {{Review}} on the {{Unit Commitment Problem}}},
  author = {Montero, Luis and Bello, Antonio and Reneses, Javier},
  year = 2022,
  month = jan,
  journal = {Energies},
  volume = {15},
  number = {4},
  pages = {1296},
  publisher = {Multidisciplinary Digital Publishing Institute},
  issn = {1996-1073},
  doi = {10.3390/en15041296},
  urldate = {2024-01-26},
  copyright = {http://creativecommons.org/licenses/by/3.0/},
  langid = {english}
}

@article{morales-espanaHiddenPowerSystem2017,
  title = {Hidden Power System Inflexibilities Imposed by Traditional Unit Commitment Formulations},
  author = {{Morales-Espa{\~n}a}, Germ{\'a}n and {Ram{\'i}rez-Elizondo}, Laura and Hobbs, Benjamin F.},
  year = 2017,
  month = apr,
  journal = {Applied Energy},
  volume = {191},
  pages = {223--238},
  issn = {0306-2619},
  doi = {10.1016/j.apenergy.2017.01.089},
  urldate = {2024-06-12}
}

@article{morales-espanaModelingHiddenFlexibility2019,
  title = {Modeling the {{Hidden Flexibility}} of {{Clustered Unit Commitment}}},
  author = {{Morales-Espana}, German and {Tejada-Arango}, Diego A.},
  year = 2019,
  month = jul,
  journal = {IEEE Transactions on Power Systems},
  volume = {34},
  number = {4},
  pages = {3294--3296},
  issn = {0885-8950, 1558-0679},
  doi = {10.1109/TPWRS.2019.2908051},
  urldate = {2023-05-16},
  langid = {english}
}

@article{morales-espanaTightCompactMILP2013,
  title = {Tight and {{Compact MILP Formulation}} for the {{Thermal Unit Commitment Problem}}},
  author = {{Morales-Espana} and Latorre, Jesus M. and Ramos, Andres},
  year = 2013,
  month = nov,
  journal = {IEEE Transactions on Power Systems},
  volume = {28},
  number = {4},
  pages = {4897--4908},
  issn = {0885-8950, 1558-0679},
  doi = {10.1109/TPWRS.2013.2251373},
  urldate = {2023-05-16},
  langid = {english}
}

@article{morales-espanaTightCompactMILP2013a,
  title = {Tight and {{Compact MILP Formulation}} of {{Start-Up}} and {{Shut-Down Ramping}} in {{Unit Commitment}}},
  author = {{Morales-Espana}, German and Latorre, Jesus M. and Ramos, Andres},
  year = 2013,
  month = may,
  journal = {IEEE Transactions on Power Systems},
  volume = {28},
  number = {2},
  pages = {1288--1296},
  issn = {0885-8950, 1558-0679},
  doi = {10.1109/TPWRS.2012.2222938},
  urldate = {2023-05-16},
  langid = {english}
}

@article{morales-espanaTightMIPFormulations2015a,
  title = {Tight {{MIP}} Formulations of the Power-Based Unit Commitment Problem},
  author = {{Morales-Espa{\~n}a}, Germ{\'a}n and Gentile, Claudio and Ramos, Andres},
  year = 2015,
  month = oct,
  journal = {OR Spectrum},
  volume = {37},
  number = {4},
  pages = {929--950},
  issn = {1436-6304},
  doi = {10.1007/s00291-015-0400-4},
  urldate = {2023-11-23},
  langid = {english}
}

@article{muckstadtApplicationLagrangianRelaxation1977,
  title = {An {{Application}} of {{Lagrangian Relaxation}} to {{Scheduling}} in {{Power-Generation Systems}}},
  author = {Muckstadt, John A. and Koenig, Sherri A.},
  year = 1977,
  journal = {Operations Research},
  volume = {25},
  number = {3},
  pages = {387--403},
  publisher = {INFORMS},
  urldate = {2025-08-06},
  langid = {english}
}

@article{nowakStochasticLagrangianRelaxation2000,
  title = {Stochastic {{Lagrangian Relaxation Applied}} to {{Power Scheduling}} in a {{Hydro-Thermal System}} under {{Uncertainty}}},
  author = {Nowak, Matthias and Roemisch, Werner},
  year = 2000,
  month = nov,
  journal = {Annals of Operations Research},
  volume = {100},
  pages = {251--272},
  doi = {10.1023/a:1019248506301}
}

@article{ostrowskiTightMixedInteger2012,
  title = {Tight {{Mixed Integer Linear Programming Formulations}} for the {{Unit Commitment Problem}}},
  author = {Ostrowski, James and Anjos, Miguel F. and Vannelli, Anthony},
  year = 2012,
  month = feb,
  journal = {IEEE Transactions on Power Systems},
  volume = {27},
  number = {1},
  pages = {39--46},
  issn = {1558-0679},
  doi = {10.1109/TPWRS.2011.2162008},
  urldate = {2024-01-26}
}

@article{palmintierHeterogeneousUnitClustering2014,
  title = {Heterogeneous {{Unit Clustering}} for {{Efficient Operational Flexibility Modeling}}},
  author = {Palmintier, Bryan S. and Webster, Mort D.},
  year = 2014,
  month = may,
  journal = {IEEE Transactions on Power Systems},
  volume = {29},
  number = {3},
  pages = {1089--1098},
  issn = {0885-8950, 1558-0679},
  doi = {10.1109/TPWRS.2013.2293127},
  urldate = {2023-05-16},
  langid = {english}
}

@article{palmintierImpactOperationalFlexibility2015,
  title = {Impact of {{Operational Flexibility}} on {{Electricity Generation Planning With Renewable}} and {{Carbon Targets}}},
  author = {Palmintier, Bryan and Webster, Mort},
  year = 2015,
  month = dec,
  journal = {IEEE Transactions on Sustainable Energy},
  volume = {7},
  pages = {1--13},
  doi = {10.1109/TSTE.2015.2498640}
}

@book{palmintierImpactUnitCommitment2011,
  title = {Impact of Unit Commitment Constraints on Generation Expansion Planning with Renewables},
  author = {Palmintier, Bryan and Webster, Mort},
  year = 2011,
  month = aug,
  journal = {Power and Energy Society General Meeting, 2011 IEEE},
  pages = {7},
  doi = {10.1109/PES.2011.6038963}
}

@misc{panConvexHullsUnit2017,
  title = {Convex {{Hulls}} for the {{Unit Commitment Polytope}}},
  author = {Pan, Kai and Guan, Yongpei},
  year = 2017,
  month = jan,
  number = {arXiv:1701.08943},
  eprint = {1701.08943},
  primaryclass = {math},
  publisher = {arXiv},
  doi = {10.48550/arXiv.1701.08943},
  urldate = {2025-07-31},
  archiveprefix = {arXiv}
}

@misc{panPolyhedralStudyIntegrated2016,
  title = {A {{Polyhedral Study}} of the {{Integrated Minimum-Up}}/-{{Down Time}} and {{Ramping Polytope}}},
  author = {Pan, Kai and Guan, Yongpei},
  year = 2016,
  month = apr,
  number = {arXiv:1604.02184},
  eprint = {1604.02184},
  publisher = {arXiv},
  doi = {10.48550/arXiv.1604.02184},
  urldate = {2024-11-11},
  archiveprefix = {arXiv}
}

@article{ponceletImpactLevelTemporal2016,
  title = {Impact of the Level of Temporal and Operational Detail in Energy-System Planning Models},
  author = {Poncelet, Kris and Delarue, Erik and Six, Daan and Duerinck, Jan and D'haeseleer, William},
  year = 2016,
  month = jan,
  journal = {Applied Energy},
  volume = {162},
  pages = {631--643},
  issn = {0306-2619},
  doi = {10.1016/j.apenergy.2015.10.100},
  urldate = {2024-01-25}
}

@article{ponceletSelectingRepresentativeDays2017,
  title = {Selecting {{Representative Days}} for {{Capturing}} the {{Implications}} of {{Integrating Intermittent Renewables}} in {{Generation Expansion Planning Problems}}},
  author = {Poncelet, Kris and H{\"o}schle, Hanspeter and Delarue, Erik and Virag, Ana and D'haeseleer, William},
  year = 2017,
  month = may,
  journal = {IEEE Transactions on Power Systems},
  volume = {32},
  number = {3},
  pages = {1936--1948},
  issn = {1558-0679},
  doi = {10.1109/TPWRS.2016.2596803},
  urldate = {2024-01-11}
}

@article{queyranneTightMIPFormulations2017,
  title = {Tight {{MIP}} Formulations for Bounded up/down Times and Interval-Dependent Start-Ups},
  author = {Queyranne, Maurice and Wolsey, Laurence A.},
  year = 2017,
  month = jul,
  journal = {Mathematical Programming},
  volume = {164},
  number = {1},
  pages = {129--155},
  issn = {1436-4646},
  doi = {10.1007/s10107-016-1079-2},
  urldate = {2024-09-04},
  langid = {english}
}

@techreport{rajanMinimumPolytopesUnit2005,
  title = {Minimum {{Up}}/{{Down Polytopes}} of the {{Unit Commitment Problem}} with {{Start-Up Costs}}},
  author = {Rajan, Deepak and Takriti, Samer},
  year = 2005,
  institution = {IBM},
  urldate = {2024-09-04},
  copyright = {\copyright{} Copyright IBM Corp. 2016},
  langid = {english}
}

@article{ridhaComplexityProfilesLargeScale2020,
  title = {Complexity {{Profiles}}: {{A Large-Scale Review}} of {{Energy System Models}} in {{Terms}} of {{Complexity}}},
  shorttitle = {Complexity {{Profiles}}},
  author = {Ridha, Elias and Nolting, Lars and Praktiknjo, Aaron},
  year = 2020,
  month = jul,
  journal = {Energy Strategy Reviews},
  volume = {30},
  pages = {100515},
  doi = {10.1016/j.esr.2020.100515}
}

@article{tejada-arangoWhichUnitCommitmentFormulation2020,
  title = {Which {{Unit-Commitment Formulation}} Is {{Best}}? {{A Comparison Framework}}},
  shorttitle = {Which {{Unit-Commitment Formulation}} Is {{Best}}?},
  author = {{Tejada-Arango}, Diego A. and Lumbreras, Sara and {S{\'a}nchez-Mart{\'i}n}, Pedro and Ramos, Andres},
  year = 2020,
  month = jul,
  journal = {IEEE Transactions on Power Systems},
  volume = {35},
  number = {4},
  pages = {2926--2936},
  issn = {1558-0679},
  doi = {10.1109/TPWRS.2019.2962024},
  urldate = {2024-09-10}
}

@article{vanackooijLargescaleUnitCommitment2018,
  title = {Large-Scale Unit Commitment under Uncertainty: An Updated Literature Survey},
  shorttitle = {Large-Scale Unit Commitment under Uncertainty},
  author = {Van Ackooij, W. and Danti Lopez, I. and Frangioni, A. and Lacalandra, F. and Tahanan, M.},
  year = 2018,
  month = dec,
  journal = {Annals of Operations Research},
  volume = {271},
  number = {1},
  pages = {11--85},
  issn = {0254-5330, 1572-9338},
  doi = {10.1007/s10479-018-3003-z},
  urldate = {2025-01-04},
  langid = {english}
}

@article{wangLongcycleStoragesAligned2023,
  title = {Long-Cycle Storages Aligned with Two-Dimensional Clustered Unit Commitment},
  author = {Wang, Jingbo and Shang, Ce},
  year = 2023,
  month = aug,
  journal = {International Journal of Electrical Power \& Energy Systems},
  volume = {150},
  pages = {109090},
  issn = {01420615},
  doi = {10.1016/j.ijepes.2023.109090},
  urldate = {2024-08-20},
  langid = {english}
}

@inproceedings{wogrinWhatTimeperiodAggregation2019,
  title = {What Time-Period Aggregation Method Works Best for Power System Operation Models with Renewables and Storage?},
  booktitle = {2019 {{International Conference}} on {{Smart Energy Systems}} and {{Technologies}} ({{SEST}})},
  author = {Wogrin, S. and {Tejada-Arango}, D.A. and Pineda, S. and Morales, J.M.},
  year = 2019,
  month = sep,
  pages = {1--6},
  doi = {10.1109/SEST.2019.8849027}
}

@article{wuijtsEffectModellingChoices2023,
  title = {Effect of Modelling Choices in the Unit Commitment Problem},
  author = {Wuijts, R.H. and {van den Akker}, M. and {van den Broek}, M.},
  year = 2023,
  journal = {Energy Systems},
  issn = {1868-3967},
  doi = {10.1007/s12667-023-00564-5},
  langid = {english}
}

@article{yangNovelProjectedTwo2016,
  title = {A {{Novel Projected Two Binary Variables Formulation}} for {{Unit Commitment Problem}}},
  author = {Yang, Linfeng and Zhang, Chen and Jian, Jinbao and Meng, Ke and Dong, Z.Y.},
  year = 2016,
  month = jun,
  journal = {Applied Energy},
  volume = {187},
  doi = {10.1016/j.apenergy.2016.11.096}
}

@article{zhangTightUnitAggregation2025,
  title = {A {{Tight Unit Aggregation}} for {{Unit Commitment}} to {{Eliminate Symmetry}}},
  author = {Zhang, Biyuan and Ding, Tao and Xiao, Yang},
  year = 2025,
  journal = {IEEE Transactions on Power Systems},
  pages = {1--11},
  issn = {1558-0679},
  doi = {10.1109/TPWRS.2025.3578915},
  urldate = {2025-06-30}
}

@misc{tianPolyhedralStudyUnit2026,
    title = {A {Polyhedral} {Study} on {Unit} {Commitment} with a {Single} {Type} of {Binary} {Variables}},
    url = {http://arxiv.org/abs/2602.22058},
    doi = {10.48550/arXiv.2602.22058},
    urldate = {2026-02-26},
    publisher = {arXiv},
    author = {Tian, Bin and Pan, Kai and Li, Chung-Lun},
    month = feb,
    year = {2026},
    note = {arXiv:2602.22058 [math]},
}
}{\scriptsize\par}

\end{document}